\documentclass[a4paper,12pt]{article}

\makeatletter
\def\normalsize{\@setfontsize\normalsize\@xipt{11}}
\makeatother
\usepackage{amssymb}
\usepackage{amsthm}
\usepackage{amsmath}
\usepackage{amsfonts}
\usepackage{setspace}  
\usepackage{anysize}
\usepackage{vmargin}
\usepackage{caption}

\setpapersize{A4} 
\setmarginsrb{4cm}{3cm}{2.5cm}{3cm}{13pt}{19pt}{13pt}{27pt}

\setpapersize{A4} 
\setmarginsrb{2.5cm}{2.5cm}{2.5cm}{2.5cm}{0pt}{0.7cm}{0pt}{0.7cm}

\theoremstyle{plain}
\newtheorem{theorem}{Theorem}
\newtheorem{lemma}{Lemma}

\theoremstyle{definition}

\theoremstyle{remark}

\begin{document}

\begin{center}
\LARGE
Asymptotic Behaviour of Resonance Eigenvalues of the Schr\"{o}dinger Operator with a Matrix Potential
\end{center}
\vspace{1cm}

\normalsize
\centerline{Sedef KARAKILI\c{C}}\centerline{Department of
Mathematics, Faculty of Science, Dokuz Eyl\"{u}l
University,} \centerline{T{\i}naztepe Camp., Buca, 35160, Izmir,
Turkey} \centerline{sedef.erim@deu.edu.tr}
\vspace{0.5cm}
\centerline{Setenay AKDUMAN} \centerline{Department of Mathematics,
	Faculty of  Science, Dokuz Eyl\"{u}l University,}
\centerline{T{\i}naztepe Camp., Buca, 35160, Izmir, Turkey}
\centerline{setenayakduman@gmail.com}
\vspace{0.5cm}
\centerline{Didem CO\d{S}KAN} \centerline{Department of
Mathematics, Faculty of Science, Dokuz Eyl\"{u}l
University,} \centerline{T{\i}naztepe Camp., Buca, 35160, Izmir,
Turkey} \centerline{coskan.didem@gmail.com}
\vspace{0.5cm}

\begin{abstract}

We will discuss the asymptotic behaviour of the eigenvalues of Schr\"{o}dinger operator with a matrix potential defined by Neumann boundary condition in $L_2^m(F)$, where $F$ is $d$-dimensional rectangle and the potential is a  $m \times m$ matrix with $m\geq 2$,  $d\geq 2$ , when the  eigenvalues belong to the resonance domain, roughly speaking they lie  near planes of diffraction. \\

\textbf{Keywords:} Schr\"{o}dinger operator, Neumann condition, Resonance eigenvalue, Perturbation theory.

\textbf{AMS Subject Classifications:} 47F05, 35P15

\end{abstract}
\newpage
\begin{center}
\large
Introduction
\end{center}
\normalsize
\par 
In this paper, we consider the Schr\"{o}dinger Operator with a matrix potential $V(x)$ defined  by the differential expression
    \begin{equation}\label{1}
    L\phi=-\Delta \phi+V\phi
    \end{equation}
     and the Neumann boundary condition
     \begin{equation}\label{2}
     \frac{\partial \phi }{\partial n} |_{\partial F}=0,
     \end{equation}
     in $L_2^m(F)$ where $F$ is the $d$ dimensional rectangle $F=[0,a_1]\times[0,a_2]\times\ldots \times [0,a_d],$ $\partial F$ is the  boundary of $F$, $m\geqslant 2, d\geqslant 2$,  $\frac{\partial}{\partial n}$ denotes differentiation along the outward normal of the boundary $\partial F$, $\boldsymbol{\Delta}$ is a diagonal $m \times m$ matrix whose diagonal elements are the scalar Laplace operators $\Delta=\frac{\partial ^2}{\partial {x_1}^2}+\frac{\partial ^2}{\partial {x_2}^2}+\ldots+\frac{\partial ^2}{\partial {x_d}^2}$ , $x=(x_1,x_2,\ldots,x_d)\in \boldsymbol{R}^d$, $V$ is a real valued symmetric matrix $V(x)=(v_{ij}(x)), i, j=1,2,\ldots,m, v_{ij}(x)\in L_2(F)$, that is, $V^T(x)=V(x).$
     \par We denote the operator defined by (1)-(2) by $L(V)$, and the eigenvalues and corresponding eigenfunctions of $L(V)$ by $\Lambda_{N}$ and $\Psi_{N}$, respectively.
\par We assume that the Fourier coefficients $v_{ij\gamma}$ of $v_{ij}(x)$ satisfy
\begin{equation}\label{coeff}
\sum_{\gamma\in\frac{\Gamma}{2}}\mid v_{ij\gamma}\mid^{2}(1+\mid
\gamma\mid^{2l})<\infty,
\end{equation}
for each $i,j=1,2,\ldots,m$,\;\;  $l>\frac{(d+20)(d-1)}{2}+d+3$
which implies
\begin{equation}\label{decomp}
v_{ij}(x)=\sum_{\gamma\in\Gamma^{+0}(\rho^{\alpha})}v_{ij\gamma}u_{\gamma}(x)+
O(\rho^{-p\alpha}),
\end{equation}
where $\Gamma^{+0}(\rho^{\alpha})=\{\gamma\in\frac{\Gamma}{2}:0\leq\mid
\gamma\mid<\rho^{\alpha}\}$, $p=l-d$, $\alpha<\frac{1}{d+20}$,
$\rho$ is a large parameter and $O(\rho^{-p\alpha})$ is a function
in $L_{2}(F)$ with norm of order $\rho^{-p\alpha}$. Furthermore, by \eqref{coeff}, we have
\begin{equation}\label{decompsum}
    M_{ij}\equiv\sum_{\gamma\in\frac{\Gamma}{2}}\mid v_{ij\gamma}\mid<\infty,
\end{equation}
for all $i,j=1,2,\ldots,m$.
\newpage
\par 
We denote the operator defined by the differential expression \eqref{1} when $V(x)=V_0$, where $V_0=\int_{F}V(x)dx$ and the boundary condition \eqref{2}, by $L(V_0).$
\par As in \cite{Veliev1}, \cite{Veliev4}, \cite{Veliev5}, we divide $R^{d}$ into two domains:
Resonance and Non-resonance domains. \\In order to define these
domains,
let us introduce the following sets:
\par Let $\alpha<\frac{1}{d+20}$, $\alpha_{k}=3^{k}\alpha$,
$k=1,2,\ldots,d-1$ and
\begin{center}
$V_{b}(\rho^{\alpha_{1}})\equiv\left\{x\in R^{d}:\quad\left\lvert\lvert
x \lvert
^{2}- \lvert x+b \lvert^{2}\right \lvert<\rho^{\alpha_{1}}\right\}$,
\end{center}
\begin{center}
$E_{1}(\rho^{\alpha_{1}},p)\equiv\bigcup\limits_{{b}\in\Gamma(p\rho^{\alpha})}$
$V_{b}(\rho^{\alpha_{1}})$,
\end{center}
\begin{center}
$U(\rho^{\alpha_{1}},p)\equiv R^{d}\setminus
E_{1}(\rho^{\alpha_{1}},p)$,
\end{center}
\begin{center}
$E_k(\rho^{\alpha_k},p)=\bigcup\limits_{\gamma_1,\gamma_2,\ldots,\gamma_k\in \Gamma(p\rho^{\alpha})}\left(\bigcap\limits_{i=1}^kV_{\gamma_i}(\rho^{\alpha_k})\right)$,
\end{center}
where $\Gamma(p\rho^{\alpha})\equiv\left\{b\in\frac{\Gamma}{2}:0<\mid
b\mid<p\rho^{\alpha}\right\}$ and the intersection $\bigcap\limits_{i=1}^kV_{\gamma_i}(\rho^{\alpha_k}) $ in $E_k$ is taken over $\gamma_1,\gamma_2,\ldots,\gamma_k$ which are linearly independent vectors and the length of $\gamma_i$ is not greater than the length of the other vector in $\Gamma \bigcap \gamma_i\boldsymbol{R}.$ The set $U(\rho^{\alpha_{1}},p)$ is said to be a non-resonance domain, and the eigenvalue $\lvert \gamma \lvert^2$ is called a non-resonance eigenvalue if $\gamma\in U(\rho^{\alpha_{1}},p).$ The domains $V_{b}(\rho^{\alpha_{1}})$, for $b\in \Gamma(p\rho^{\alpha})$ are called resonance domains and the eigenvalue $\lvert \gamma \lvert^2$ is a resonance eigenvalue if $\gamma \in V_{b}(\rho^{\alpha_{1}})$.
\par As noted in \cite{Veliev4} and \cite{Veliev5}, the domain $V_{b}(\rho^{\alpha_{1}})\setminus E_2$, called a single resonance domain, has asymptotically full measure on $V_{b}(\rho^{\alpha_{1}}),$ that is,
\begin{align*}
\frac{\mu\left(\left(V_{b}(\rho^{\alpha_{1}})\setminus E_2\right) \bigcap B(q)\right) }{\mu\left(V_{b}(\rho^{\alpha_{1}})\bigcap B(q) \right)}\to 1,\text{ as  } \rho\to \infty,
\end{align*}
where $B(\rho)=\left\{x\in \boldsymbol{R}^d:\lvert x \lvert =\rho \right\}$, if
\begin{align}\label{alfa}
2\alpha_2-\alpha_1+(d+3)\alpha<1 \text{  and  } \alpha_2>2\alpha_1,
\end{align}
hold. Since $\alpha<\frac{1}{d+20}$, the conditions in \eqref{alfa} hold.
\par In \cite{coskan1}, the asymptotic formulas of arbitrary order for the non-resonance eigenvalues of the Schr\"{o}dinger operator $L(V)$ with the condition \eqref{coeff} are obtained.
 \par In this paper, we obtain the  asymptotic formulas for the resonance eigenvalues.
 The main result of this paper is to find connection between the eigenvalues of the Schr\"{o}dinger operator corresponding to a single resonance domain and the eigenvalues of the Sturm-Liouville operators.
 \par
 The organization of this paper is as follows:
 \par
 In the first part of this paper, we obtained asymptotic formulas depend not only the eigenvalues of the matric $C(\gamma,\gamma_1,\gamma_2,\ldots,\gamma_k)$ but also on the eigenvalues of the matrix $V_0$.
 \par
 In the second part, we investigated the perturbation of the eigenvalue $|\gamma|^2$ when $\gamma\in V_{\delta}(\rho^{\alpha_1})\backslash E_2.$ We assume that $\gamma=(\gamma^1,\gamma^2,\ldots,\gamma^d)\notin V_{e_k}(\rho^{\alpha_1})$ for $k=1,2,\ldots,d,$ where 
 $e_1=\left(\frac{\pi}{a_1},0,\ldots,0\right), \\e_2=\left(0,\frac{\pi}{a_2},\ldots,0\right), \ldots,e_d=\left(0,\ldots,\frac{\pi}{a_d}\right).$ This relation implies that $\gamma^k>\frac{1}{3\rho^{\alpha_1}},$ $\forall k=1,2,\ldots,d.$
\par
The case $\delta=e_i$, $i=1,2,\ldots,d$ was considered in \cite{Karakilicsetenay}, where a different method was used. Since there is no intersection between two investigated methods of the cases $\delta=e_i$ and $\delta\neq e_i$, $i=1,2,\ldots,d$, we study them in different papers.
\section{Asymptotic Formulas for the Eigenvalues in the Resonance Domain}

\hspace{0.5 cm} We assume that $\gamma\notin
V_{e_{k}}(\rho^{\alpha_{1}})$ for $k=1,2,\ldots,d$ where
$e_{1}=(\frac{\pi}{a_{1}},0,\ldots,0),$\\
$e_{2}=(0,\frac{\pi}{a_{2}},0,\ldots,0),\ldots,
e_{d}=(0,\ldots,0,\frac{\pi}{a_{d}})$.

\par Let $\mid\gamma\mid^{2}$ be a resonance
eigenvalue of the operator $L(0)$, that is,
$\gamma\in(\bigcap\limits_{i=1}^{k}V_{\gamma_{i}}(\rho^{\alpha_{k}}))\setminus
E_{k+1}$, $k=1,2,\ldots,d-1$, $\gamma_{i}\neq e_{j}$ for
$i=1,2,\ldots,k$ and $j=1,2,\ldots,d-1$.

We define the following sets
\begin{equation*}
B_{k}(\gamma_{1},\gamma_{2},\ldots,\gamma_{k})=\{b:
b=\sum\limits_{i=1}^{k}n_{i}\gamma_{i},n_{i}\in Z, \mid
b\mid<\frac{1}{2}\rho^{\frac{1}{2}\alpha_{k+1}}\},
\end{equation*}
\begin{equation*}
B_{k}(\gamma)=\gamma+B_{k}(\gamma_{1},\gamma_{2},\ldots,\gamma_{k})
=\{\gamma+b: b\in B_{k}(\gamma_{1},\gamma_{2},\ldots,\gamma_{k})\},
\end{equation*}
\begin{equation*}
B_{k}(\gamma,p_{1})=B_{k}(\gamma)+\Gamma(p_{1}\rho^{\alpha}).
\end{equation*}

Let $h_{\tau}$, $\tau=1,2,\ldots,b_{k}$ denote the vectors of
$B_{k}(\gamma,p_{1})$, $b_{k}$ the number of the vectors in
$B_{k}(\gamma,p_{1})$. We define the $mb_{k}\times mb_{k}$ matrix
$C=C(\gamma,\gamma_{1},\ldots,\gamma_{k})$ by
\begin{equation}\label{appmat}
C=\left[
  \begin{array}{cccc}
    \mid h_{1}\mid^{2} I & V_{h_{1}-h_{2}} & \cdots & V_{h_{1}-h_{b_{k}}}\\
    V_{h_{2}-h_{1}} & \mid h_{2}\mid^{2} I & \cdots& V_{h_{2}-h_{b_{k}}} \\
    \vdots &  &  &  \\
    V_{h_{b_{k}}-h_{1}} & V_{h_{b_{k}}-h_{2}} & \cdots & \mid h_{b_{k}}\mid^{2} I\\
  \end{array}
\right],
\end{equation}
where $V_{h_{\tau}-h_{\xi}}$, $\tau,\xi=1,2,\ldots,b_{k}$ are the
$m\times m$ matrices defined by
\begin{equation}\label{matrixdecompold}
V_{h_{\tau}-h_{\xi}}=\left[
  \begin{array}{cccc}
    v_{11h_{\tau}-h_{\xi}} & v_{12h_{\tau}-h_{\xi}} & \cdots & v_{1mh_{\tau}-h_{\xi}} \\
    v_{21h_{\tau}-h_{\xi}} & v_{22h_{\tau}-h_{\xi}} & \cdots & v_{2mh_{\tau}-h_{\xi}} \\
    \vdots &  &  &  \\
    v_{m1h_{\tau}-h_{\xi}} & v_{m2h_{\tau}-h_{\xi}} & \cdots & v_{mmh_{\tau}-h_{\xi}} \\
  \end{array}
\right].
\end{equation}

The analogues of the following lemma can be found in \cite{Karakilic3}. (see Theorem 3.1.1.)
\begin{lemma}\label{lemres_res1}
Let $\mid\gamma\mid^{2}$ be a resonance eigenvalue of the operator
$L(0)$, that is,
$\gamma\in(\bigcap\limits_{i=1}^{k}V_{\gamma_{i}}(\rho^{\alpha_{k}}))\setminus
E_{k+1}$, $k=1,2,\ldots,d-1$ where $\mid\gamma\mid\sim\rho$,
$\Lambda_{N}$ an eigenvalue of the operator $L(V)$ satisfying
\begin{equation}\label{lemres_res1eq1}
\mid\Lambda_{N}-\mid\gamma\mid^{2}\mid<\frac{1}{2}\rho^{\alpha_{1}}.
\end{equation}
Then
\begin{equation}\label{lemres_res1eq2}
    \mid\Lambda_{N}-\mid h_{\tau}-\gamma\prime-\gamma_{1}-\gamma_{2}-\cdots-\gamma_{s}\mid^{2}\mid>\frac{1}{6}\rho^{\alpha_{k+1}}
\end{equation}
where $h_{\tau}\in B_{k}(\gamma,p_{1})$,
$h_{\tau}-\gamma\prime\notin B_{k}(\gamma,p_{1})$,
$\gamma\prime\in\Gamma(\rho^{\alpha})$,
$\gamma_{i}\in\Gamma(\rho^{\alpha})$, $i=1,2,\ldots,s$,\\
$s=0,1,\ldots,p_{1}-1$.
\end{lemma}
\begin{proof} 
The relations $h_{\tau}\in B_{k}(\gamma,p_{1})$, $h_{\tau}-\gamma\prime\notin B_{k}(\gamma,p_{1})$,
$2p_{1}>p$ and
$\mid\gamma\prime\mid,\mid\gamma_{1}\mid,\ldots,$\\$\mid\gamma_{p_{1}-1}\mid<\rho^{\alpha}$
imply that
\begin{equation*}
    a_{s}=h_{\tau}-\gamma\prime-\gamma_{1}-\gamma_{2}-\ldots-\gamma_{s}\in B_{k}(\gamma,p_{1})\setminus B_{k}(\gamma)
\end{equation*}
for $s=0,1,\ldots,p_{1}-1$. To prove the inequality
\eqref{lemres_res1eq2}, we use the decomposition
\begin{equation*}
    a_{s}=\gamma+b+a,
\end{equation*}
where $b\in B_{k}$ and $a\in \Gamma(p_{1}\rho^{\alpha})$. So $\mid
b\mid<\frac{1}{2}\rho^{\frac{1}{2}\alpha_{k+1}}$ and $\mid
a\mid<p_{1}\rho^{\alpha}$. First we show that
\begin{equation}\label{lemres_res1eq3}
    \mid\mid\gamma+b+a\mid^{2}-\mid\gamma\mid^{2}\mid>\frac{1}{5}\rho^{\alpha_{k+1}}.
\end{equation}
To prove the inequality \eqref{lemres_res1eq3}, we consider the following cases.\\
Case 1: If $a\in P=span\{\gamma_{1},\gamma_{2},\ldots,\gamma_{k}\}$,
then $a+b\in P$ and $\gamma+b+a\notin B_{k}(\gamma)$ imply that
$a+b\in P\setminus B_{k}$, that is,
\begin{equation*}
    \mid a+b\mid\geq\frac{1}{2}\rho^{\frac{1}{2}\alpha_{k+1}}.
\end{equation*}
Now, if we consider the orthogonal decomposition of $\gamma$ as
$\gamma=x+v$ where $v\in P$ and $x\bot v$, then by using $x\cdot
a=x\cdot b=x\cdot v=0$, $\mid
a+b\mid\geq\frac{1}{2}\rho^{\frac{1}{2}\alpha_{k+1}}$ and $\mid
v\mid<\rho^{\alpha_{1}}$, we get
\begin{eqnarray*}
  \mid\mid\gamma+b+a\mid^{2}-\mid\gamma\mid^{2}\mid &=& \mid\mid x+v+b+a\mid^{2}-\mid x+v\mid^{2}\mid \\
   &=&  \mid\mid v+b+a\mid^{2}-\mid v\mid^{2}\mid >\frac{1}{5}\rho^{\frac{1}{2}\alpha_{k+1}}.
\end{eqnarray*}
Thus for Case 1 the inequality \eqref{lemres_res1eq3} is true.\\
Case 2: If $a\notin P$, then by definition of
$\gamma\in(\bigcap\limits_{i=1}^{k}V_{\gamma_{i}}(\rho^{\alpha_{k}}))\setminus
E_{k+1}$, we have
\begin{equation}\label{lemres_res1eq4}
    \mid\mid\gamma+a\mid^{2}-\mid\gamma\mid^{2}\mid>\rho^{\alpha_{k+1}}.
\end{equation}
Consider the difference
\begin{equation*}
    \mid\mid\gamma+b+a\mid^{2}-\mid\gamma\mid^{2}\mid=
    \mid\mid\gamma+b+a\mid^{2}-\mid\gamma+b\mid^{2}+\mid\gamma+b\mid^{2}-\mid\gamma\mid^{2}\mid,
\end{equation*}
where
\begin{center}
    $d_{1}=\mid\gamma+b+a\mid^{2}-\mid\gamma+b\mid^{2}$, $d_{2}=\mid\gamma+b\mid^{2}-\mid\gamma\mid^{2}$.
\end{center}
Since
\begin{equation*}
    d_{1}=\mid\gamma+b+a\mid^{2}-\mid\gamma+b\mid^{2}=
    \mid\gamma+a\mid^{2}-\mid\gamma\mid^{2}+2a\cdot b,
\end{equation*}
by the inequality \eqref{lemres_res1eq4} and $\mid2a\cdot
b\mid\leq2\mid a\mid\mid
b\mid<p_{1}\rho^{\alpha}\rho^{\frac{1}{2}\alpha_{k+1}}<\frac{1}{3}\rho^{\alpha_{k+1}}$,
\begin{equation*}
    \mid d_{1}\mid>\frac{2}{3}\rho^{\alpha_{k+1}}.
\end{equation*}
On the other hand, using
$\mid\gamma+b+a\mid^{2}-\mid\gamma\mid^{2}=\mid v+b+a\mid^{2}-\mid
v\mid^{2}$, and taking $a=0$, we get
\begin{equation*}
    d_{2}=\mid\gamma+b\mid^{2}-\mid\gamma\mid^{2}=\mid v+b\mid^{2}-\mid v\mid^{2}=
    (\mid v+b\mid-\mid v\mid)(\mid v+b\mid+\mid v\mid)
\end{equation*}
from which it follows that
\begin{equation*}
    \mid d_{2}\mid<\frac{1}{3}\rho^{\alpha_{k+1}}.
\end{equation*}
Then
\begin{equation*}
    \mid\mid d_{1}\mid-\mid d_{2}\mid\mid>\frac{1}{5}\rho^{\alpha_{k+1}}.
\end{equation*}
So in any case the inequality \eqref{lemres_res1eq3} is true.
Therefore, the inequalities \eqref{lemres_res1eq1} and
\eqref{lemres_res1eq3} imply that
\begin{equation*}
    \mid\Lambda_{N}-\mid\gamma+b+a\mid^{2}\mid=
    \mid\Lambda_{N}-\mid\gamma\mid^{2}-\mid\gamma+b+a\mid^{2}+\mid\gamma\mid^{2}\mid>\frac{1}{6}\rho^{\alpha_{k+1}}.
\end{equation*}
\end{proof}

\begin{theorem}\label{theorem_res1}
Let $\mid\gamma\mid^{2}$ be a resonance eigenvalue of the operator
$L(0)$, that is,
$\gamma\in(\bigcap\limits_{i=1}^{k}V_{\gamma_{i}}(\rho^{\alpha_{k}}))\setminus
E_{k+1}$, $k=1,2,\ldots,d-1$ where $\mid\gamma\mid\sim\rho$,
$\lambda_{i}$ an eigenvalue of the matrix $V_{0}$, and $\Lambda_{N}$
an eigenvalue of the operator $L(V)$ satisfying
\begin{equation}\label{theorem_res1eq100}
\mid\Lambda_{N}-\mid\gamma\mid^{2}\mid<\frac{1}{2}\rho^{\alpha_{1}}
\end{equation}
and
\begin{equation}\label{theorem_res1eq111}
\mid<\Phi_{\gamma,j},\Psi_{N}>\mid>c_{17}\rho^{-c\alpha}.
\end{equation}
Then there exists an eigenvalue $\eta_{s}(\gamma)$,
$s=1,2,\ldots,mb_{k}$ of the matrix $C$ such that
\begin{equation*}
\Lambda_{N}=\lambda_{i}+\eta_{s}(\gamma)+O(\rho^{-(p-c-\frac{d}{4}3^{d})\alpha}).
\end{equation*}
\end{theorem}
\begin{proof}
We give the proof by using the same consideration as in Karak{\i}l{\i}\d{c} (2004). 
The binding formula 
\begin{equation}\label{binding}
(\Lambda_N-|h_{\tau}|^2)\langle\Psi_N,\Phi_{h_{\tau},j}\rangle=\langle\Psi_N,V\Phi_{h_{\tau},j}\rangle
\end{equation}
for any $h_{\tau}\in
B_{k}(\gamma,p_{1})$, $\tau=1,2,\ldots,b_{k}$ and the decomposition
\begin{eqnarray}\label{decomp11}
\nonumber V(x)\Phi_{h_{\tau},j}(x) &=&
(\sum\limits_{\gamma\prime\in\Gamma^{+0}(\rho^{\alpha})}v_{1j\gamma\prime}u_{h_{\tau}+\gamma\prime}(x),\ldots,
\sum\limits_{\gamma\prime\in\Gamma^{+0}(\rho^{\alpha})}v_{mj\gamma\prime}u_{h_{\tau}+\gamma\prime}(x))
+O(\rho^{-p\alpha})\\
&=&
\sum\limits_{i=1}^{m}\sum\limits_{\gamma\prime\in\Gamma^{+0}(\rho^{\alpha})}
v_{ij\gamma\prime}\Phi_{h_{\tau}+\gamma\prime,i}(x)
+O(\rho^{-p\alpha}).
\end{eqnarray}
give
\begin{equation}\label{theorem_res1eq1}
(\Lambda_{N}-\mid h_{\tau}\mid^{2})<\Psi_{N},\Phi_{h_{\tau},j}>=
\sum\limits_{i=1}^{m}\sum\limits_{\gamma\prime\in\Gamma^{+0}(\rho^{\alpha})}
v_{ij\gamma\prime}<\Psi_{N},\Phi_{h_{\tau}-\gamma\prime,i}>
+O(\rho^{-p\alpha}).
\end{equation}
We first show that
\begin{equation}\label{theorem_res1eq2}
   O(\rho^{-p\alpha})=\sum\limits_{i=1}^{m}\sum\limits_{\gamma\prime\in\Gamma(\rho^{\alpha})
   \atop{h_{\tau}-\gamma\prime\notin B_{k}(\gamma,p_{1})}}
    v_{ij\gamma\prime}<\Psi_{N},\Phi_{h_{\tau}-\gamma\prime,i}>
\end{equation}
for any $j=1,2,\ldots,m$. Here we remark that $\gamma\prime\neq0$.
If it were the case, then we would have from
$h_{\tau}-\gamma\prime\notin B_{k}(\gamma,p_{1})$ that $h_{\tau}\notin B_{k}(\gamma,p_{1})$ which is a contradiction.\\
Since $\Lambda_{N}$ satisfies the inequality
\eqref{theorem_res1eq100}, by Lemma \ref{lemres_res1}, we have
$\mid\Lambda_{N}-\mid
h_{\tau}-\gamma\prime\mid^{2}\mid>\frac{1}{6}\rho^{\alpha_{k+1}}$.
Using this and the decomposition \eqref{theorem_res1eq1} for
$h_{\tau}-\gamma\prime\notin B_{k}(\gamma,p_{1})$, it follows that
\begin{eqnarray*}
 &&  \sum\limits_{i=1}^{m}\sum\limits_{\gamma\prime\in\Gamma(\rho^{\alpha})
     \atop{h_{\tau}-\gamma\prime\notin B_{k}(\gamma,p_{1})}}
    v_{ij\gamma\prime}<\Psi_{N},\Phi_{h_{\tau}-\gamma\prime,i}>\\
    &=& \sum\limits_{i=1}^{m}\sum\limits_{\gamma\prime\in\Gamma(\rho^{\alpha})
    \atop{h_{\tau}-\gamma\prime\notin B_{k}(\gamma,p_{1})}}
    \frac{v_{ij\gamma\prime}}{\Lambda_{N}-\mid h_{\tau}-\gamma\prime\mid^{2}}
    \sum\limits_{i_{1}=1}^{m}\sum\limits_{\gamma_{1}\in\Gamma(\rho^{\alpha})
    \atop{h_{\tau}-\gamma\prime\notin B_{k}(\gamma,p_{1})}}
    v_{i_{1}i\gamma_{1}}<\Psi_{N},\Phi_{h_{\tau}-\gamma\prime-\gamma_{1},i_{1}}>+ O(\rho^{-p\alpha}) .
\end{eqnarray*}
In this manner, iterating $p_{1}$ times, we get
\begin{eqnarray*}
&&\sum\limits_{i=1}^{m}\sum\limits_{\gamma\prime\in\Gamma(\rho^{\alpha})\atop{h_{\tau}-\gamma\prime\notin
B_{k}(\gamma,p_{1})}}
    v_{ij\gamma\prime}<\Psi_{N},\Phi_{h_{\tau}-\gamma\prime,i}>\\
  &=& \sum\limits_{i,i_{1},i_{2},\ldots,i_{p_{1}}=1}^{m}
  \sum\limits_{\gamma\prime,\gamma_{1},\gamma_{2},\ldots,\gamma_{p_{1}}\in\Gamma(\rho^{\alpha})
  \atop{h_{\tau}-\gamma\prime\notin B_{k}(\gamma,p_{1})}}\\
  &&\frac{v_{ij\gamma\prime}v_{i_{1}i\gamma_{1}}\ldots v_{i_{p_{1}}i_{p_{1}-1}\gamma_{p_{1}}}<\Psi_{N},\Phi_{h_{\tau}-\gamma\prime-\gamma_{1}-\cdots-\gamma_{p_{1}},i_{p_{1}}}>}
  {(\Lambda_{N}-\mid h_{\tau}-\gamma\prime\mid^{2})(\Lambda_{N}-\mid h_{\tau}-\gamma\prime-\gamma_{1}\mid^{2})
  \ldots (\Lambda_{N}-\mid h_{\tau}-\gamma\prime-\gamma_{1}-\cdots-\gamma_{p_{1}-1}\mid^{2})} 
  +O(\rho^{-p\alpha}).
\end{eqnarray*}
Taking norm of both sides of the last equality, using Lemma
\ref{lemres_res1}, the relation 
\begin{eqnarray}\label{M}
M_{ij}=\sum_{\gamma\in\frac{\Gamma}{2}}|v_{ij\gamma}|<\infty
\end{eqnarray}
and the fact that
$p_{1}\alpha_{k+1}\geq p_{1}\alpha_{2}>p\alpha$, we obtain
\begin{eqnarray*}
&&\mid\sum\limits_{i=1}^{m}\sum\limits_{\gamma\prime\in\Gamma(\rho^{\alpha})\atop{h_{\tau}-\gamma\prime\notin
B_{k}(\gamma,p_{1})}}
    v_{ij\gamma\prime}<\Psi_{N},\Phi_{h_{\tau}-\gamma\prime,i}>\mid\\
  &\leq& \sum\limits_{i,i_{1},i_{2},\ldots,i_{p_{1}}=1}^{m}
  \sum\limits_{\gamma\prime,\gamma_{1},\gamma_{2},\ldots,\gamma_{p_{1}}\in\Gamma(\rho^{\alpha})
  \atop{h_{\tau}-\gamma\prime\notin B_{k}(\gamma,p_{1})}} \\
  & &\frac{\mid v_{ij\gamma\prime}\mid\mid v_{i_{1}i\gamma_{1}}\mid\ldots\mid v_{i_{p_{1}}i_{p_{1}-1}\gamma_{p_{1}}}\mid  \mid<\Psi_{N},\Phi_{h_{\tau}-\gamma\prime-\gamma_{1}-\cdots-\gamma_{p_{1}},i_{p_{1}}}>\mid  }
  {\mid\Lambda_{N}-\mid h_{\tau}-\gamma\prime\mid^{2}\mid\mid\Lambda_{N}-\mid h_{\tau}-\gamma\prime-\gamma_{1}\mid^{2}\mid
  \ldots \mid\Lambda_{N}-\mid h_{\tau}-\gamma\prime-\gamma_{1}-\cdots-\gamma_{p_{1}-1}\mid^{2}\mid}
  +O(\rho^{-p\alpha}) \\
    &\leq& (\frac{1}{6}\rho^{\alpha_{k+1}})^{-p_{1}}
  \sum\limits_{\gamma\prime,\gamma_{1},\gamma_{2},\ldots,\gamma_{p_{1}}\in\Gamma(\rho^{\alpha})
  \atop{h_{\tau}-\gamma\prime\notin B_{k}(\gamma,p_{1})}}
  \mid v_{ij\gamma\prime}\mid\mid v_{i_{1}i\gamma_{1}}\mid\ldots\mid v_{i_{p_{1}}i_{p_{1}-1}\gamma_{p_{1}}}\mid \mid<\Psi_{N},\Phi_{h_{\tau}-\gamma\prime-\gamma_{1}-\cdots-\gamma_{p_{1}},i_{p_{1}}}>\mid
  +O(\rho^{-p\alpha}) \\
  &\leq& (\frac{1}{6}\rho^{\alpha_{k+1}})^{-p_{1}} \sum\limits_{i,i_{1},i_{2},\ldots,i_{p_{1}}=1}^{m}
  M_{ij}M_{i_{1}i}\ldots M_{i_{p_{1}i_{p_{1}-1}}}
  \mid<\Psi_{N},\Phi_{h_{\tau}-\gamma\prime-\gamma_{1}-\cdots-\gamma_{p_{1}},i_{p_{1}}}>\mid \\
  &+&O(\rho^{-p\alpha}) \\
  &=& O(\rho^{-p\alpha}).
\end{eqnarray*}
That is, the estimation \eqref{theorem_res1eq2} holds. Therefore,
the decomposition \eqref{theorem_res1eq1} becomes
\begin{equation}\label{theorem_res1eq3}
(\Lambda_{N}-\mid h_{\tau}\mid^{2})<\Psi_{N},\Phi_{h_{\tau},j}>=
\sum\limits_{i=1}^{m}\sum\limits_{\gamma\prime\in\Gamma^{+0}(\rho^{\alpha})\atop{h_{\tau}-\gamma\prime\in
B_{k}(\gamma,p_{1})}}
v_{ij\gamma\prime}<\Psi_{N},\Phi_{h_{\tau}-\gamma\prime,i}>
+O(\rho^{-p\alpha}).
\end{equation}
Since $h_{\tau}-\gamma\prime\in B_{k}(\gamma,p_{1})$, using the
notation $h_{\xi}=h_{\tau}-\gamma\prime$, the decomposition
\eqref{theorem_res1eq3} can be written as
\begin{equation*}
  (\Lambda_{N}-\mid h_{\tau}\mid^{2})<\Psi_{N},\Phi_{h_{\tau},j}> =
  \sum\limits_{i=1}^{m}\sum\limits_{h_{\tau}-h_{\xi}\in\Gamma^{+0}(\rho^{\alpha})}
  v_{ijh_{\tau}-h_{\xi}}<\Psi_{N},\Phi_{h_{\xi},i}>
  +O(\rho^{-p\alpha}).
\end{equation*}
Isolating the terms where $h_{\tau}-h_{\xi}=0$, we get
\begin{eqnarray}\label{theorem_res1eq4}
  \nonumber (\Lambda_{N}-\mid h_{\tau}\mid^{2})<\Psi_{N},\Phi_{h_{\tau},j}> &=&
  \sum\limits_{i=1}^{m}v_{ij0}<\Psi_{N},\Phi_{h_{\tau},i}> \\
  \nonumber &+& \sum\limits_{i=1}^{m}\sum\limits_{h_{\tau}-h_{\xi}\in\Gamma(\rho^{\alpha})}
  v_{ijh_{\tau}-h_{\xi}}<\Psi_{N},\Phi_{h_{\xi},i}> \\
  &+&O(\rho^{-p\alpha}).
\end{eqnarray}
Considering the decomposition \eqref{theorem_res1eq4} for an
arbitrary $h_{\tau}\in B_{k}(\gamma,p_{1})$, $\tau=1,2,\ldots,b_{k}$
and for all $j=1,2,\ldots,m$, we get
\begin{equation}\label{theorem_res1eq5}
    (\Lambda_{N}-\mid h_{\tau}\mid^{2})I A(N,h_{\tau})=V_{0}A(N,h_{\tau})+
    \sum\limits_{\xi=1\atop{\xi\neq\tau}}^{b_k}V_{h_{\tau}-h_{\xi}}A(N,h_{\xi})+O(\rho^{-p\alpha}),
\end{equation}
or
\begin{equation}\label{theorem_res1eq6}
    [(\Lambda_{N}-\mid h_{\tau}\mid^{2})I -V_{0}] A(N,h_{\tau})=
    \sum\limits_{\xi=1\atop{\xi\neq\tau}}^{b_k}V_{h_{\tau}-h_{\xi}}A(N,h_{\xi})+O(\rho^{-p\alpha}),
\end{equation}
where $I$ is an $m\times m$ identity matrix, $V_{h_{\tau}-h_{\xi}}$
is given by \eqref{matrixdecompold}, $O(\rho^{-p\alpha})$ is an
$m\times1$ vector and $A(N,h_{\xi})$ is the $m\times1$ vector
\begin{equation}\label{matrixinnerbk}
A(N,h_{\xi})=(<\Psi_{N},\Phi_{h_{\xi},1}>,
<\Psi_{N},\Phi_{h_{\xi},2}>,\ldots,<\Psi_{N},\Phi_{h_{\xi},m}>)
\end{equation}
for any $\xi=1,2,\ldots,b_{k}$. \\
Let $\lambda_{i}$ be an eigenvalue of the matrix $V_{0}$ and
$\omega_{i}$ the corresponding normalized eigenvector. Multiplying
both sides of the decomposition \eqref{theorem_res1eq6} by
$\omega_{i}$, we get
\begin{equation}\label{theorem_res1eq7}
    [(\Lambda_{N}-\mid h_{\tau}\mid^{2})I -V_{0}] A(N,h_{\tau})\cdot\omega_{i}=
    \sum\limits_{\xi=1\atop{\xi\neq\tau}}^{b_k}V_{h_{\tau}-h_{\xi}}A(N,h_{\xi})\cdot\omega_{i}+O(\rho^{-p\alpha}).
\end{equation}
For the left hand side of this last equality we have
\begin{eqnarray}\label{theorem_res1eq8}
  \nonumber[(\Lambda_{N}-\mid h_{\tau}\mid^{2})I -V_{0}] A(N,h_{\tau})\cdot\omega_{i}
  &=&
  \nonumber A(N,h_{\tau})\cdot[(\Lambda_{N}-\mid h_{\tau}\mid^{2})I -V_{0}]\omega_{i} \\
  &=&
  \nonumber A(N,h_{\tau})\cdot(\Lambda_{N}-\mid h_{\tau}\mid^{2}-\lambda_{i})\omega_{i} \\
  &=&
  (\Lambda_{N}-\mid h_{\tau}\mid^{2}-\lambda_{i})A(N,h_{\tau})\cdot\omega_{i}.
\end{eqnarray}
Letting $\lambda_{N,\tau,i}=\Lambda_{N}-\mid
h_{\tau}\mid^{2}-\lambda_{i}$, by the equation
\eqref{theorem_res1eq8}, we have from the decomposition
\eqref{theorem_res1eq7} that
\begin{equation}\label{theorem_res1eq9}
    [\lambda_{N,\tau,i}IA(N,h_{\tau})-
    \sum\limits_{\xi=1\atop{\xi\neq\tau}}^{b_k}V_{h_{\tau}-h_{\xi}}A(N,h_{\xi})]\cdot\omega_{i}
    =O(\rho^{-p\alpha}).
\end{equation}
Since the set of normalized eigenvectors
$\{\omega_{i}\}_{i=1,2,\ldots,m}$ of the matrix $V_0$ forms a basis
for $R^{m}$, for any vector
$\lambda_{N,\tau,i}IA(N,h_{\tau})-\sum\limits_{\xi=1\atop{\xi\neq\tau}}^{b_k}V_{h_{\tau}-h_{\xi}}A(N,h_{\xi})$,
$\tau=1,2,\ldots,b_k$ in $R^{m}$ by using Parseval's relation and
the equation \eqref{theorem_res1eq9}, we have
\begin{flushleft}
    $\mid
    \lambda_{N,\tau,i}IA(N,h_{\tau})-\sum\limits_{\xi=1\atop{\xi\neq\tau}}^{b_k}V_{h_{\tau}-h_{\xi}}A(N,h_{\xi})
    \mid^{2}$
\begin{equation}\label{theorem_res1eq14}
    =\sum\limits_{i=1}^{m}
    \mid
    [\lambda_{N,\tau,i}IA(N,h_{\tau})-
    \sum\limits_{\xi=1\atop{\xi\neq\tau}}^{b_k}V_{h_{\tau}-h_{\xi}}A(N,h_{\xi})]\cdot\omega_{i}
    \mid^{2}
    =\sum\limits_{i=1}^{m}\mid O(\rho^{-p\alpha})\mid^{2}.
\end{equation}
\end{flushleft}
It follows from \eqref{theorem_res1eq14} that
\begin{equation}\label{theorem_res1eq10}
    \lambda_{N,\tau,i}IA(N,h_{\tau})-
    \sum\limits_{\xi=1\atop{\xi\neq\tau}}^{b_k}V_{h_{\tau}-h_{\xi}}A(N,h_{\xi})
    =O(\rho^{-p\alpha}).
\end{equation}
Now, considering the equation \eqref{theorem_res1eq10} for all
$h_{\tau}\in B_{k}(\gamma,p_{1})$, $\tau=1,2,\ldots,b_{k}$, we
obtain the system
\begin{equation}\label{theorem_res1eq11}
\left[
  \begin{array}{cccc}
    \lambda_{N,1,i}I & -V_{h_{1}-h_{2}} & \cdots & -V_{h_{1}-h_{b_{k}}}\\
    -V_{h_{2}-h_{1}} & \lambda_{N,2,i}I & \cdots& -V_{h_{2}-h_{b_{k}}} \\
    \vdots &  &  &  \\
    -V_{h_{b_{k}}-h_{1}} & -V_{h_{b_{k}}-h_{2}} & \cdots & \lambda_{N,b_{k},i}I\\
  \end{array}
\right] \left[
\begin{array}{c}
  A(N,h_{1}) \\
  A(N,h_{2}) \\
  \vdots \\
  A(N,h_{b_{k}})
\end{array}
\right] = \left[
\begin{array}{c}
  O(\rho^{-p\alpha}) \\
  O(\rho^{-p\alpha})\\
  \vdots \\
  O(\rho^{-p\alpha})
\end{array}
\right].
\end{equation}
We may write the system \eqref{theorem_res1eq11} as
\begin{equation}\label{theorem_res1eq12}
    [(\Lambda_{N}-\lambda_{i})I-C]A(N,h_{1},h_{2},\ldots,h_{b_{k}})=O(\rho^{-p\alpha}),
\end{equation}
where $I$ is an $mb_k\times mb_k$ identity matrix, $C$ is given by
\eqref{appmat}, $A(N,h_{1},h_{2},\ldots,h_{b_{k}})$ is the
$mb_{k}\times1$ vector
\begin{equation*}
    A(N,h_{1},h_{2},\ldots,h_{b_{k}})=(A(N,h_{1}),A(N,h_{2}),\ldots,A(N,h_{b_{k}}))
\end{equation*}
and $O(\rho^{-p\alpha})$ is an $mb_{k}\times1$ vector. Multiplying
both sides of the equation \eqref{theorem_res1eq12} by
$[(\Lambda_{N}-\lambda_{i})I-C]^{-1}$, and taking norm of both
sides, we get
\begin{equation}\label{theorem_res1eq13}
    \mid A(N,h_{1},h_{2},\ldots,h_{b_{k}})\mid
    \leq
    \parallel[(\Lambda_{N}-\lambda_{i})I-C]^{-1}\parallel\mid O(\rho^{-p\alpha})\mid.
\end{equation}
By the estimation \eqref{theorem_res1eq111}, together with
$b_{k}=O(\rho^{\frac{d}{2}3^{d}\alpha})$ we have the estimations
\begin{center}
    $\mid A(N,h_{1},h_{2},\ldots,h_{b_{k}})\mid>c_{18}\rho^{-c\alpha},\quad$
    $\mid O(\rho^{-p\alpha})\mid=O(\rho^{-(p-\frac{d}{4}3^{d})\alpha}).$
\end{center}
Thus it follows from the inequality \eqref{theorem_res1eq13} and the
last estimations that
\begin{equation*}
    c_{18}\rho^{-c\alpha}\leq
    \parallel[(\Lambda_{N}-\lambda_{i})I-C]^{-1}\parallel
    c_{19}\rho^{-(p-\frac{d}{4}3^{d})\alpha},
\end{equation*}
\begin{equation*}
    \min_{s=1,2,\ldots,mb_k}\mid\Lambda_{N}-\lambda_{i}-\eta_{s}(\gamma)\mid\leq c_{20}\rho^{-(p-c-\frac{d}{4}3^{d})\alpha},
\end{equation*}
\begin{equation*}
    \Lambda_{N}=\lambda_{i}+\eta_{s}(\gamma)+O(\rho^{-(p-c-\frac{d}{4}3^{d})\alpha}).
\end{equation*}
\end{proof}

\begin{theorem}\label{theorem_res2}
Let $\mid\gamma\mid^{2}$ be a resonance eigenvalue of the operator
$L(0)$, that is,
$\gamma\in(\bigcap\limits_{i=1}^{k}V_{\gamma_{i}}(\rho^{\alpha_{k}}))\setminus
E_{k+1}$, $k=1,2,\ldots,d-1$ where $\mid\gamma\mid\sim\rho$,
$\lambda_{i}$ an eigenvalue of the matrix $V_{0}$,
$\eta_{s}(\gamma)$ an eigenvalue of the matrix $C$ such that
$\mid\eta_{s}(\gamma)-\mid\gamma\mid^{2}\mid<\frac{3}{8}\rho^{\alpha_{1}}$.
Then there is an eigenvalue $\Lambda_{N}$ of the operator $L(V)$
satisfying
\begin{equation}\label{theorem_res2eq1}
    \Lambda_{N}=\lambda_{i}+\eta_{s}(\gamma)+O(\rho^{-p\alpha+\frac{d}{4}3^{d}\alpha+\frac{d-1}{2}}),
\end{equation}
where $\eta_{s}(\gamma)$, $s=1,2,\ldots,mb_{k}$ is an eigenvalue of
the matrix $C$ which is given by \eqref{appmat}.
\end{theorem}
\begin{proof}
We prove this theorem by using the same consideration as in Karak{\i}l{\i}\d{c} (2004). 
By the general perturbation theory, there is an eigenvalue
$\Lambda_{N}$ of the operator $L(V)$ such that
$\mid\Lambda_{N}-\mid\gamma\mid^{2}\mid<\frac{1}{2}\rho^{2\alpha_{1}}$ holds.
Thus one can use the system \eqref{theorem_res1eq12}
\begin{equation}\label{theorem_res2eq2}
    [(\Lambda_{N}-\lambda_{i})I-C]A(N,h_{1},h_{2},\ldots,h_{b_{k}})=O(\rho^{-p\alpha})
\end{equation}
of Theorem \ref{theorem_res1}. Let $\eta_{s}$, $s=1,2,\ldots,mb_{k}$
be an eigenvalue of the matrix $C$ and
$\theta_{s}=(\theta_{s}^{1},\theta_{s}^{2},\ldots,\theta_{s}^{b_k})_{mb_k\times1}$
the corresponding normalized eigenvector, $\mid \theta_{s}\mid=1$, where
$\theta_{s}^{\tau}=(\theta_{s}^{\tau1},\theta_{s}^{\tau2},\ldots,\theta_{s}^{\tau
m})_{m\times1}$, $\tau=1,2,\ldots,b_k$. Multiplying the
equation \eqref{theorem_res2eq2} by $\theta_{s}$, we get
\begin{equation}\label{theorem_res2eq4}
    [(\Lambda_{N}-\lambda_{i})I-C]A(N,h_{1},h_{2},\ldots,h_{b_{k}})\cdot \theta_{s}=O(\rho^{-p\alpha})\cdot \theta_{s}.
\end{equation}
From the left hand side of the equation \eqref{theorem_res2eq4} we
get
\begin{flushleft}
        $[(\Lambda_{N}-\lambda_{i})I-C]A(N,h_{1},h_{2},\ldots,h_{b_{k}})\cdot \theta_{s}$
\begin{eqnarray}\label{theorem_res2eq18}
  \nonumber &=&
   A(N,h_{1},h_{2},\ldots,h_{b_{k}})\cdot[(\Lambda_{N}-\lambda_{i})I-C]\theta_{s}\\
  \nonumber &=&
   A(N,h_{1},h_{2},\ldots,h_{b_{k}})\cdot[(\Lambda_{N}-\lambda_{i})I\theta_{s}-\eta_{s}\theta_{s}]\\
  \nonumber &=&
   A(N,h_{1},h_{2},\ldots,h_{b_{k}})\cdot(\Lambda_{N}-\lambda_{i}-\eta_{s})\theta_{s} \\
   &=&
   (\Lambda_{N}-\lambda_{i}-\eta_{s})A(N,h_{1},h_{2},\ldots,h_{b_{k}})\cdot \theta_{s}.
\end{eqnarray}
\end{flushleft}
Using the equation \eqref{theorem_res2eq18} in the equation
\eqref{theorem_res2eq4}, and taking norm of both sides, we get
\begin{equation}\label{theorem_res2eq5}
    \mid\Lambda_{N}-\lambda_{i}-\eta_{s}\mid\mid A(N,h_{1},h_{2},\ldots,h_{b_{k}})\cdot \theta_{s}\mid=
    \mid O(\rho^{-p\alpha})\cdot \theta_{s}\mid.
\end{equation}
From the right hand side of the equation \eqref{theorem_res2eq5} by
using $b_{k}=O(\rho^{\frac{d}{2}3^{d}\alpha})$, we have
\begin{equation}\label{theorem_res2eq6}
    \mid O(\rho^{-p\alpha})\cdot \theta_{s}\mid \leq
    \mid O(\rho^{-p\alpha})\mid\mid \theta_{s}\mid =
    \sqrt{mb_{k}(\rho^{-p\alpha})^{2}}=\sqrt{mb_{k}}\rho^{-p\alpha}=O(\rho^{-p\alpha+\frac{d}{4}3^{d}\alpha}).
\end{equation}
The equation \eqref{theorem_res2eq5} and the estimation
\eqref{theorem_res2eq6} give
\begin{equation}\label{theorem_res2eq10}
    \mid\Lambda_{N}-\lambda_{i}-\eta_{s}\mid\mid A(N,h_{1},h_{2},\ldots,h_{b_{k}})\cdot \theta_{s}\mid=
    O(\rho^{-p\alpha+\frac{d}{4}3^{d}\alpha}).
\end{equation}
Now, we estimate $\mid A(N,h_{1},h_{2},\ldots,h_{b_{k}})\cdot
\theta_{s}\mid$. Since
\begin{equation}\label{theorem_res2eq11}
    \mid A(N,h_{1},h_{2},\ldots,h_{b_{k}})\cdot \theta_{s}\mid
    =\mid\sum\limits_{\tau=1}^{b_{k}}\sum\limits_{i=1}^{m}\theta_{s}^{\tau i}<\Psi_{N},\Phi_{h_{\tau},i}>\mid
    =\mid<\Psi_{N},\sum\limits_{\tau=1}^{b_{k}}\sum\limits_{i=1}^{m}\theta_{s}^{\tau i}\Phi_{h_{\tau},i}>\mid,
\end{equation}
to estimate $\mid A(N,h_{1},h_{2},\ldots,h_{b_{k}})\cdot
\theta_{s}\mid$, we consider the Parseval's relation
\begin{eqnarray}\label{theorem_res2eq12}
\nonumber 1
   &=&
   \parallel\sum\limits_{\tau=1}^{b_{k}}\sum\limits_{i=1}^{m}\theta_{s}^{\tau i}\Phi_{h_{\tau},i}\parallel^{2}
   = \sum\limits_{N=1}^{\infty}
   \mid<\Psi_{N},
   \sum\limits_{\tau=1}^{b_{k}}\sum\limits_{i=1}^{m}\theta_{s}^{\tau i}\Phi_{h_{\tau},i}>\mid^{2} \\
\nonumber &=&
   \sum\limits_{N=1}^{\infty}
   \mid\sum\limits_{\tau=1}^{b_{k}}\sum\limits_{i=1}^{m}\theta_{s}^{\tau i}<\Psi_{N},\Phi_{h_{\tau},i}>\mid^{2} \\
\nonumber &=&
   \sum\limits_{N:\mid\Lambda_{N}-\mid\gamma\mid^{2}\mid\geq\frac{1}{2}\rho^{2\alpha_{1}}}
   \mid\sum\limits_{\tau=1}^{b_{k}}\sum\limits_{i=1}^{m}\theta_{s}^{\tau i}<\Psi_{N},\Phi_{h_{\tau},i}>\mid^{2} \\
   &+&
   \sum\limits_{N:\mid\Lambda_{N}-\mid\gamma\mid^{2}\mid<\frac{1}{2}\rho^{2\alpha_{1}}}
   \mid\sum\limits_{\tau=1}^{b_{k}}\sum\limits_{i=1}^{m}\theta_{s}^{\tau i}<\Psi_{N},\Phi_{h_{\tau},i}>\mid^{2}.
\end{eqnarray}
We give an estimation for the first summation in the last
expression.
\begin{flushleft}
    $
    \sum\limits_{N:\mid\Lambda_{N}-\mid\gamma\mid^{2}\mid\geq\frac{1}{2}\rho^{2\alpha_{1}}}
    \mid\sum\limits_{\tau=1}^{b_{k}}\sum\limits_{i=1}^{m}\theta_{s}^{\tau i}<\Psi_{N},\Phi_{h_{\tau},i}>\mid^{2}
    $
\begin{eqnarray}\label{theorem_res2eq13}
   \nonumber&=&
   \sum\limits_{N:\mid\Lambda_{N}-\mid\gamma\mid^{2}\mid\geq\frac{1}{2}\rho^{2\alpha_{1}}}
   \mid\sum\limits_{\tau:\mid\eta_{s}-\mid h_{\tau}\mid^{2}\mid<\frac{1}{8}\rho^{\alpha_{1}}}
   \sum\limits_{i=1}^{m}\theta_{s}^{\tau i}<\Psi_{N},\Phi_{h_{\tau},i}> \\
   \nonumber&+&
   \sum\limits_{\tau:\mid\eta_{s}-\mid h_{\tau}\mid^{2}\mid\geq\frac{1}{8}\rho^{\alpha_{1}}}
   \sum\limits_{i=1}^{m}\theta_{s}^{\tau i}<\Psi_{N},\Phi_{h_{\tau},i}>\mid^{2}  \\
   \nonumber&<&
   2\sum\limits_{N:\mid\Lambda_{N}-\mid\gamma\mid^{2}\mid\geq\frac{1}{2}\rho^{2\alpha_{1}}}
   \mid\sum\limits_{\tau:\mid\eta_{s}-\mid h_{\tau}\mid^{2}\mid<\frac{1}{8}\rho^{\alpha_{1}}}
   \sum\limits_{i=1}^{m}\theta_{s}^{\tau i}<\Psi_{N},\Phi_{h_{\tau},i}>\mid^{2}\\
   &+&
   2\sum\limits_{N:\mid\Lambda_{N}-\mid\gamma\mid^{2}\mid\geq\frac{1}{2}\rho^{2\alpha_{1}}}
   \mid\sum\limits_{\tau:\mid\eta_{s}-\mid h_{\tau}\mid^{2}\mid\geq\frac{1}{8}\rho^{\alpha_{1}}}
   \sum\limits_{i=1}^{m}\theta_{s}^{\tau i}<\Psi_{N},\Phi_{h_{\tau},i}>\mid^{2}.
\end{eqnarray}
\end{flushleft}
To estimate the term 
$$2\sum\limits_{N:\mid\Lambda_{N}-\mid\gamma\mid^{2}\mid\geq\frac{1}{2}\rho^{2\alpha_{1}}}
   \mid\sum\limits_{\tau:\mid\eta_{s}-\mid h_{\tau}\mid^{2}\mid\geq\frac{1}{8}\rho^{\alpha_{1}}}
   \sum\limits_{i=1}^{m}\theta_{s}^{\tau i}<\Psi_{N},\Phi_{h_{\tau},i}>\mid^{2}$$
in the inequality \eqref{theorem_res2eq13}, we consider the matrix
$C$ as $C=A+B$ where
\begin{equation}\label{theorem_res2eq7}
   A= \left[
      \begin{array}{ccc}
        \mid h_{1}\mid^{2}I &  & 0 \\
         & \ddots &  \\
        0 &  & \mid h_{b_{k}}\mid^{2}I \\
      \end{array}
    \right],
    \quad
    B =\left[
      \begin{array}{cccc}
        0 & V_{h_{1}-h_{2}} & \cdots & V_{h_{1}-h_{b_{k}}} \\
        V_{h_{2}-h_{1}} & 0 & \cdots & V_{h_{2}-h_{b_{k}}} \\
        \vdots &  & \ddots & \vdots \\
        V_{h_{b_{k}}-h_{1}} & V_{h_{b_{k}}-h_{2}} & \cdots & 0 \\
      \end{array}
    \right].
\end{equation}
Let $\{e_{\tau, i}\}_{\tau=1,2,\ldots,b_k, i=1,2,\ldots,m}$ be a set
of orthonormal vectors such that $e_{\tau, i}\cdot e_{\xi, k}=1$ if
$\tau=\xi$, $i=k$, $e_{\tau, i}\cdot e_{\xi, k}=0$ otherwise.
Multiplying $C\theta_{s}=(A+B)\theta_{s}$ by $e_{\tau, i}$, we get
\begin{equation*}
    C\theta_{s}\cdot e_{\tau, i}=(\eta_{s}\theta_{s})\cdot e_{\tau, i}=\eta_{s}(\theta_{s}\cdot e_{\tau, i})=
    \eta_{s}\theta_{s}^{\tau i},
\end{equation*}
and
\begin{equation*}
    (A+B)\theta_{s}\cdot e_{\tau, i}=\theta_{s}\cdot(A+B)e_{\tau, i}=
    \theta_{s}\cdot Ae_{\tau, i}+\theta_{s}\cdot Be_{\tau, i}=
    \theta_{s}^{\tau i}\mid h_{\tau}\mid^{2}+\theta_{s}\cdot Be_{\tau, i}.
\end{equation*}
From the equality of the last two equations we have
\begin{equation}\label{theorem_res2eq8}
    (\eta_{s}-\mid h_{\tau}\mid^{2})\theta_{s}^{\tau i}=\theta_{s}\cdot Be_{\tau, i}
\end{equation}
for any $\tau=1,2,\ldots,b_{k}$, $i=1,2,\ldots,m$.\\
Using Bessel's inequality, Parseval's relation, orthogonality of the
functions $\Phi_{h_{\tau},i}(x)$, $\tau=1,2,\ldots,b_{k}$,
$i=1,2,\ldots,m$, the binding formula \eqref{theorem_res2eq8} and
$\parallel B\parallel\leq M$, we have
\begin{flushleft}
   $
   2\sum\limits_{N:\mid\Lambda_{N}-\mid\gamma\mid^{2}\mid\geq\frac{1}{2}\rho^{2\alpha_{1}}}
   \mid\sum\limits_{\tau:\mid\eta_{s}-\mid h_{\tau}\mid^{2}\mid\geq\frac{1}{8}\rho^{\alpha_{1}}}
   \sum\limits_{i=1}^{m}\theta_{s}^{\tau i}<\Psi_{N},\Phi_{h_{\tau},i}>\mid^{2}
   $
   \begin{eqnarray}\label{theorem_res2eq14}
   \nonumber
   &=&
   2\sum\limits_{N:\mid\Lambda_{N}-\mid\gamma\mid^{2}\mid\geq\frac{1}{2}\rho^{2\alpha_{1}}}
   \mid<\Psi_{N},\sum\limits_{\tau:\mid\eta_{s}-\mid h_{\tau}\mid^{2}\mid\geq\frac{1}{8}\rho^{\alpha_{1}}}
   \sum\limits_{i=1}^{m}
   \theta_{s}^{\tau i}\Phi_{h_{\tau},i}>\mid^{2}  \\
   \nonumber
   &\leq&
   2\sum\limits_{N=1}^{\infty}
   \mid<\Psi_{N},\sum\limits_{\tau:\mid\eta_{s}-\mid h_{\tau}\mid^{2}\mid\geq\frac{1}{8}\rho^{\alpha_{1}}}
   \sum\limits_{i=1}^{m}
   \theta_{s}^{\tau i}\Phi_{h_{\tau},i}>\mid^{2} \\
   \nonumber
   &=&
   2\parallel\sum\limits_{\tau:\mid\eta_{s}-\mid h_{\tau}\mid^{2}\mid\geq\frac{1}{8}\rho^{\alpha_{1}}}
   \sum\limits_{i=1}^{m}
   \theta_{s}^{\tau i}\Phi_{h_{\tau},i}\parallel^{2} \\
   \nonumber
   &=&
   2\sum\limits_{\tau:\mid\eta_{s}-\mid h_{\tau}\mid^{2}\mid\geq\frac{1}{8}\rho^{\alpha_{1}}}
   \sum\limits_{i=1}^{m}
   \mid \theta_{s}^{\tau i}\mid^{2}\parallel\Phi_{h_{\tau},i}\parallel^{2}
   \\ \nonumber
   &=&
   2\sum\limits_{\tau:\mid\eta_{s}-\mid h_{\tau}\mid^{2}\mid\geq\frac{1}{8}\rho^{\alpha_{1}}}
   \sum\limits_{i=1}^{m}
   \mid \theta_{s}^{\tau i}\mid^{2} \\
   \nonumber
   &=&
   2\sum\limits_{\tau:\mid\eta_{s}-\mid h_{\tau}\mid^{2}\mid\geq\frac{1}{8}\rho^{\alpha_{1}}}
   \sum\limits_{i=1}^{m}
   \frac{\mid \theta_{s}\cdot Be_{\tau, i}\mid^{2}}{\mid \eta_{s}-\mid h_{\tau}\mid^{2}\mid^{2}} \\
   \nonumber
   &\leq&
   2(\frac{1}{8}\rho^{\alpha_{1}})^{-2}
   \sum\limits_{\tau:\mid\eta_{s}-\mid h_{\tau}\mid^{2}\mid\geq\frac{1}{8}\rho^{\alpha_{1}}}
   \sum\limits_{i=1}^{m}
   \mid \theta_{s}\mid^{2}\parallel B\parallel^{2}\mid e_{\tau, i}\mid^{2} \\
   &=&
   O(\rho^{-2\alpha_{1}}).
\end{eqnarray}
 \end{flushleft}
Now, we estimate the term
$$2\sum\limits_{N:\mid\Lambda_{N}-\mid\gamma\mid^{2}\mid\geq\frac{1}{2}\rho^{2\alpha_{1}}}
   \mid\sum\limits_{\tau:\mid\eta_{s}-\mid h_{\tau}\mid^{2}\mid<\frac{1}{8}\rho^{\alpha_{1}}}
   \sum\limits_{i=1}^{m}\theta_{s}^{\tau i}<\Psi_{N},\Phi_{h_{\tau},i}>\mid^{2}$$
in the inequality \eqref{theorem_res2eq13}. The assumption
$$\mid\eta_{s}-\mid\gamma\mid^{2}\mid<\frac{3}{8}\rho^{\alpha_{1}}$$
of the theorem together with
$\mid\Lambda_{N}-\mid\gamma\mid^{2}\mid\geq\frac{1}{2}\rho^{2\alpha_{1}}$
and $\mid\eta_{s}-\mid
h_{\tau}\mid^{2}\mid<\frac{1}{8}\rho^{\alpha_{1}}$ imply that
$\mid\Lambda_{N}-\mid
h_{\tau}\mid^{2}\mid>\frac{1}{2}\rho^{\alpha_{1}}$ and
$\mid\mid\gamma\mid^{2}-\mid
h_{\tau}\mid^{2}\mid<\frac{1}{2}\rho^{\alpha_{1}}$. So one has
\begin{eqnarray}\label{theorem_res2eq9}
  \nonumber\frac{1}{\Lambda_{N}-\mid h_{\tau}\mid^{2}}
  &=&
  \frac{1}{\Lambda_{N}-\mid\gamma\mid^{2}}
  \sum\limits_{n=0}^{\infty}
  (\frac{\mid h_{\tau}\mid^{2}-\mid\gamma\mid^{2}}{\Lambda_{N}-\mid\gamma\mid^{2}})^{n}\\
  &=&
  \frac{1}{\Lambda_{N}-\mid\gamma\mid^{2}}
  \{\sum\limits_{n=0}^{k}
  (\frac{\mid h_{\tau}\mid^{2}-\mid\gamma\mid^{2}}{\Lambda_{N}-\mid\gamma\mid^{2}})^{n}+O(\rho^{-(k+1)\alpha_{1}})\}.
\end{eqnarray}
Using the binding formula \eqref{binding} for any $h_{\tau}\in
B_{k}(\gamma,p_{1})$, $\mid\Lambda_{N}-\mid
h_{\tau}\mid^{2}\mid>\frac{1}{2}\rho^{\alpha_{1}}$ and the
decomposition \eqref{theorem_res2eq9}, we have
\begin{eqnarray*}
 & & 2\sum\limits_{N:\mid\Lambda_{N}-\mid\gamma\mid^{2}\mid\geq\frac{1}{2}\rho^{2\alpha_{1}}}
   \mid\sum\limits_0{\tau:\mid\eta_{s}-\mid h_{\tau}\mid^{2}\mid<\frac{1}{8}\rho^{\alpha_{1}}}
   \sum\limits_{i=1}^{m}\theta_{s}^{\tau i}<\Psi_{N},\Phi_{h_{\tau},i}>\mid^{2}\\
&=& 2\sum\limits_{N:\mid\Lambda_{N}-\mid\gamma\mid^{2}\mid\geq\frac{1}{2}\rho^{2\alpha_{1}}}
    \mid
    \sum\limits_{\tau:\mid\eta_{s}-\mid h_{\tau}\mid^{2}\mid<\frac{1}{8}\rho^{\alpha_{1}}}
    \sum\limits_{i=1}^{m}\theta_{s}^{\tau i}\frac{<\Psi_{N},V\Phi_{h_{\tau},i}>}{\Lambda_{N}-\mid h_{\tau}\mid^{2}}
    \mid^{2} \\
    &=& 2\sum\limits_{N:\mid\Lambda_{N}-\mid\gamma\mid^{2}\mid\geq\frac{1}{2}\rho^{2\alpha_{1}}}
    \mid
    \sum\limits_{\tau:\mid\eta_{s}-\mid h_{\tau}\mid^{2}\mid<\frac{1}{8}\rho^{\alpha_{1}}}\sum\limits_{i=1}^{m}
    \frac{\theta_{s}^{\tau i}<\Psi_{N},V\Phi_{h_{\tau},i}>}{\Lambda_{N}-\mid\gamma\mid^{2}}  \left\{\sum\limits_{n=0}^{k}
    (\frac{\mid h_{\tau}\mid^{2}-\mid\gamma\mid^{2}}{\Lambda_{N}-\mid\gamma\mid^{2}})^{n}+O(\rho^{-(k+1)\alpha_{1}})\right\}
    \mid^{2}\\
    &\leq& 2\sum\limits_{N:\mid\Lambda_{N}-\mid\gamma\mid^{2}\mid\geq\frac{1}{2}\rho^{2\alpha_{1}}}(k+1)
    \mid
    \sum\limits_{\tau:\mid\eta_{s}-\mid h_{\tau}\mid^{2}\mid<\frac{1}{8}\rho^{\alpha_{1}}}\sum\limits_{i=1}^{m}
    \frac{\theta_{s}^{\tau i}<\Psi_{N},V\Phi_{h_{\tau},i}>}{\Lambda_{N}-\mid\gamma\mid^{2}}
    \mid^{2} \\
    &+& 2\sum\limits_{N:\mid\Lambda_{N}-\mid\gamma\mid^{2}\mid\geq\frac{1}{2}\rho^{2\alpha_{1}}}(k+1)
    \mid
    \sum\limits_{\tau:\mid\eta_{s}-\mid h_{\tau}\mid^{2}\mid<\frac{1}{8}\rho^{\alpha_{1}}}\sum\limits_{i=1}^{m}
    \frac{\theta_{s}^{\tau i}<\Psi_{N},V\Phi_{h_{\tau},i}>}{\Lambda_{N}-\mid\gamma\mid^{2}}
    \frac{\mid h_{\tau}\mid^{2}-\mid\gamma\mid^{2}}{\Lambda_{N}-\mid\gamma\mid^{2}}
    \mid^{2} \\
    &\vdots&  \\
    &+& 2\sum\limits_{N:\mid\Lambda_{N}-\mid\gamma\mid^{2}\mid\geq\frac{1}{2}\rho^{2\alpha_{1}}}(k+1)
    \mid
    \sum\limits_{\tau:\mid\eta_{s}-\mid h_{\tau}\mid^{2}\mid<\frac{1}{8}\rho^{\alpha_{1}}}\sum\limits_{i=1}^{m}
    \frac{\theta_{s}^{\tau i}<\Psi_{N},V\Phi_{h_{\tau},i}>}{\Lambda_{N}-\mid\gamma\mid^{2}}
    [\frac{\mid h_{\tau}\mid^{2}-\mid\gamma\mid^{2}}{\Lambda_{N}-\mid\gamma\mid^{2}}]^{k}
    \mid^{2} \\
    &+& 2\sum\limits_{N:\mid\Lambda_{N}-\mid\gamma\mid^{2}\mid\geq\frac{1}{2}\rho^{2\alpha_{1}}}(k+1)
    \mid
    \sum\limits_{\tau:\mid\eta_{s}-\mid h_{\tau}\mid^{2}\mid<\frac{1}{8}\rho^{\alpha_{1}}}\sum\limits_{i=1}^{m}
    \frac{\theta_{s}^{\tau i}<\Psi_{N},V\Phi_{h_{\tau},i}>}{\Lambda_{N}-\mid\gamma\mid^{2}}O(\rho^{-(k+1)\alpha_{1}})
    \mid^{2}.
\end{eqnarray*}
We estimate
\begin{equation*}
    2\sum\limits_{N:\mid\Lambda_{N}-\mid\gamma\mid^{2}\mid\geq\frac{1}{2}\rho^{2\alpha_{1}}}(k+1)
    \mid
    \sum\limits_{\tau:\mid\eta_{s}-\mid h_{\tau}\mid^{2}\mid<\frac{1}{8}\rho^{\alpha_{1}}}
    \sum\limits_{i=1}^{m}\theta_{s}^{\tau i}<\Psi_{N},V\Phi_{h_{\tau},i}>
    \frac{(\mid h_{\tau}\mid^{2}-\mid\gamma\mid^{2})^{r}}{(\Lambda_{N}-\mid\gamma\mid^{2})^{r+1}}
    \mid^{2},
\end{equation*}
where $r=0,1,2,\ldots,k$ and
\begin{equation*}
    2\sum\limits_{N:\mid\Lambda_{N}-\mid\gamma\mid^{2}\mid\geq\frac{1}{2}\rho^{2\alpha_{1}}}(k+1)
    \mid
    \sum\limits_{\tau:\mid\eta_{s}-\mid h_{\tau}\mid^{2}\mid<\frac{1}{8}\rho^{\alpha_{1}}}\sum\limits_{i=1}^{m}
    \frac{\theta_{s}^{\tau i}<\Psi_{N},V\Phi_{h_{\tau},i}>}{\Lambda_{N}-\mid\gamma\mid^{2}}
    O(\rho^{-(k+1)\alpha_{1}})
    \mid^{2}.
\end{equation*}
For an arbitrary $r=0,1,2,\ldots,k$ using Bessel's inequality,
triangle inequality, $$\mid\theta_{s}^{\tau i}\mid\leq1,\quad
\mid\mid\gamma\mid^{2}-\mid
h_{\tau}\mid^{2}\mid<\frac{1}{2}\rho^{\alpha_{1}}$$ and the relations
\eqref{M} and 
$$M_i=\sum_{j=1}^mM_{ij},\quad M_j=\sum_{i=1}^mM_{ij}\quad
M^2=\max_{1\leq i\leq m}  M_{i}\max_{1\leq j \leq m}M_{j},$$
we have
\begin{eqnarray*}
&&2\sum\limits_{N:\mid\Lambda_{N}-\mid\gamma\mid^{2}\mid\geq\frac{1}{2}\rho^{2\alpha_{1}}}(k+1)
    \mid
    \sum\limits_{\tau:\mid\eta_{s}-\mid h_{\tau}\mid^{2}\mid<\frac{1}{8}\rho^{\alpha_{1}}}\sum\limits_{i=1}^{m}
    \theta_{s}^{\tau i}<\Psi_{N},V\Phi_{h_{\tau},i}>
    \frac{(\mid h_{\tau}\mid^{2}-\mid\gamma\mid^{2})^{r}}{(\Lambda_{N}-\mid\gamma\mid^{2})^{r+1}}
    \mid^{2}\\ 
   &=&
   2\sum\limits_{N:\mid\Lambda_{N}-\mid\gamma\mid^{2}\mid\geq\frac{1}{2}\rho^{2\alpha_{1}}}
   \frac{(k+1)}{\mid\Lambda_{N}-\mid\gamma\mid^{2}\mid^{2(r+1)}}
   \mid
   \sum\limits_{\tau:\mid\eta_{s}-\mid h_{\tau}\mid^{2}\mid<\frac{1}{8}\rho^{\alpha_{1}}}\sum\limits_{i=1}^{m}
   \theta_{s}^{\tau i}<\Psi_{N},V\Phi_{h_{\tau},i}>
   (\mid h_{\tau}\mid^{2}-\mid\gamma\mid^{2})^{r}
   \mid^{2} \\
   &\leq&
   2(\frac{1}{2}\rho^{2\alpha_{1}})^{-2(r+1)}(k+1)
   \sum\limits_{N:\mid\Lambda_{N}-\mid\gamma\mid^{2}\mid\geq\frac{1}{2}\rho^{2\alpha_{1}}}
   \mid
   <\Psi_{N},
   \sum\limits_{\tau:\mid\eta_{s}-\mid h_{\tau}\mid^{2}\mid<\frac{1}{8}\rho^{\alpha_{1}}}
   \sum\limits_{i=1}^{m}
   \theta_{s}^{\tau i}
   (\mid h_{\tau}\mid^{2}-\mid\gamma\mid^{2})^{r}
   V\Phi_{h_{\tau},i}>
   \mid^{2} \\
   &\leq&
   2(\frac{1}{2}\rho^{2\alpha_{1}})^{-2(r+1)}(k+1)
   \parallel
   \sum\limits_{\tau:\mid\eta_{s}-\mid h_{\tau}\mid^{2}\mid<\frac{1}{8}\rho^{\alpha_{1}}}
   \sum\limits_{i=1}^{m}
   \theta_{s}^{\tau i}
   (\mid h_{\tau}\mid^{2}-\mid\gamma\mid^{2})^{r}
   V\Phi_{h_{\tau},i}
   \parallel^{2}
   \\
   &\leq&
   2(\frac{1}{2}\rho^{2\alpha_{1}})^{-2(r+1)}
   (k+1)
   (\sum\limits_{\tau:\mid\eta_{s}-\mid h_{\tau}\mid^{2}\mid<\frac{1}{8}\rho^{\alpha_{1}}}
   \sum\limits_{i=1}^{m}
   \parallel
   \theta_{s}^{\tau i}
   (\mid h_{\tau}\mid^{2}-\mid\gamma\mid^{2})^{r}
   V\Phi_{h_{\tau},i}>
   \parallel)^{2} \\
   &=&
   2(\frac{1}{2}\rho^{2\alpha_{1}})^{-2(r+1)}
   (k+1)
   (\sum\limits_{\tau:\mid\eta_{s}-\mid h_{\tau}\mid^{2}\mid<\frac{1}{8}\rho^{\alpha_{1}}}
   \sum\limits_{i=1}^{m}
   \mid\theta_{s}^{\tau i}\mid
   \mid\mid h_{\tau}\mid^{2}-\mid\gamma\mid^{2}\mid^{r}
   \parallel V\Phi_{h_{\tau},i}\parallel
   )^{2} \\
   &\leq&
   2(\frac{1}{2}\rho^{2\alpha_{1}})^{-2(r+1)}
   (\frac{1}{2}\rho^{\alpha_{1}})^{2r}
   (k+1)
   (\sum\limits_{\tau:\mid\eta_{s}-\mid h_{\tau}\mid^{2}\mid<\frac{1}{8}\rho^{\alpha_{1}}}\sum\limits_{i=1}^{m}
   \parallel
   V\Phi_{h_{\tau},i}
   \parallel)^{2}.
\end{eqnarray*}
Thus
\begin{eqnarray}\label{theorem_res2eq15}
\nonumber O(\rho^{-4\alpha_{1}})
    &=&
    \sum\limits_{r=0}^{k}2\sum\limits_{N:\mid\Lambda_{N}-\mid\gamma\mid^{2}\mid>\frac{1}{2}\rho^{2\alpha_{1}}}(k+1) \\
    & &
    \mid
    \sum\limits_{\tau:\mid\eta_{s}-\mid h_{\tau}\mid^{2}\mid<\frac{1}{8}\rho^{\alpha_{1}}}\sum\limits_{i=1}^{m}
    \theta_{s}^{\tau i}<\Psi_{N},V\Phi_{h_{\tau},i}>
    \frac{(\mid h_{\tau}\mid^{2}-\mid\gamma\mid^{2})^{r}}{(\Lambda_{N}-\mid\gamma\mid^{2})^{r+1}}
    \mid^{2}.
\end{eqnarray}
Similarly, we have
\begin{flushleft}
$
    2\sum\limits_{N:\mid\Lambda_{N}-\mid\gamma\mid^{2}\mid\geq\frac{1}{2}\rho^{2\alpha_{1}}}(k+1)
    \mid
    \sum\limits_{\tau:\mid\eta_{s}-\mid h_{\tau}\mid^{2}\mid<\frac{1}{8}\rho^{\alpha_{1}}}\sum\limits_{i=1}^{m}
    \frac{\theta_{s}^{\tau i}<\Psi_{N},V\Phi_{h_{\tau},i}>}{\Lambda_{N}-\mid\gamma\mid^{2}}
    O(\rho^{-(k+1)\alpha_{1}})
    \mid^{2}
$
\begin{eqnarray*}
   &=&
   2\sum\limits_{N:\mid\Lambda_{N}-\mid\gamma\mid^{2}\mid\geq\frac{1}{2}\rho^{2\alpha_{1}}}
   \frac{(k+1)}{\mid\Lambda_{N}-\mid\gamma\mid^{2}\mid^{2}} \\
   & &
   \mid
   \sum\limits_{\tau:\mid\eta_{s}-\mid h_{\tau}\mid^{2}\mid<\frac{1}{8}\rho^{\alpha_{1}}}\sum\limits_{i=1}^{m}
   \theta_{s}^{\tau i}<\Psi_{N},V\Phi_{h_{\tau},i}>
   O(\rho^{-(k+1)\alpha_{1}})
   \mid^{2} \\
   &\leq&
   2(\frac{1}{2}\rho^{2\alpha_{1}})^{-2}(k+1) \\
   & &
   \sum\limits_{N:\mid\Lambda_{N}-\mid\gamma\mid^{2}\mid\geq\frac{1}{2}\rho^{2\alpha_{1}}}
   \mid
   <\Psi_{N},
   \sum\limits_{\tau:\mid\eta_{s}-\mid h_{\tau}\mid^{2}\mid<\frac{1}{8}\rho^{\alpha_{1}}}\sum\limits_{i=1}^{m}
   \theta_{s}^{\tau i}
   O(\rho^{-(k+1)\alpha_{1}})
   V\Phi_{h_{\tau},i}>
   \mid^{2}\\
   &\leq&
   2(\frac{1}{2}\rho^{2\alpha_{1}})^{-2}(k+1)
   \parallel
   \sum\limits_{\tau:\mid\eta_{s}-\mid h_{\tau}\mid^{2}\mid<\frac{1}{8}\rho^{\alpha_{1}}}\sum\limits_{i=1}^{m}
   \theta_{s}^{\tau i}O(\rho^{-(k+1)\alpha_{1}})V\Phi_{h_{\tau},i}
   \parallel^{2}\\
   &\leq&
   2(\frac{1}{2}\rho^{2\alpha_{1}})^{-2}(k+1)
   (\sum\limits_{\tau:\mid\eta_{s}-\mid h_{\tau}\mid^{2}\mid<\frac{1}{8}\rho^{\alpha_{1}}}\sum\limits_{i=1}^{m}
   \parallel
   \theta_{s}^{\tau i}
   O(\rho^{-(k+1)\alpha_{1}})
   V\Phi_{h_{\tau},i}>
   \parallel)^{2} \\
   &\leq&
   2(\frac{1}{2}\rho^{2\alpha_{1}})^{-2}O(\rho^{-2(k+1)\alpha_{1}})(k+1)
   (\sum\limits_{\tau:\mid\eta_{s}-\mid h_{\tau}\mid^{2}\mid<\frac{1}{8}\rho^{\alpha_{1}}}\sum\limits_{i=1}^{m}
   \parallel
   V\Phi_{h_{\tau},i}
   \parallel)^{2}.
\end{eqnarray*}
\end{flushleft}
Thus
\begin{eqnarray}\label{theorem_res2eq16}
\nonumber  O(\rho^{-8\alpha_{1}})
  &=&
  2\sum\limits_{N:\mid\Lambda_{N}-\mid\gamma\mid^{2}\mid\geq\frac{1}{2}\rho^{2\alpha_{1}}}(k+1)
  \\
  & &
  \mid\sum\limits_{\tau:\mid\eta_{s}-\mid h_{\tau}\mid^{2}\mid<\frac{1}{8}\rho^{\alpha_{1}}}\sum\limits_{i=1}^{m}
  \frac{\theta_{s}^{\tau i}<\Psi_{N},V\Phi_{h_{\tau},i}>}{\Lambda_{N}-\mid\gamma\mid^{2}}
  O(\rho^{-(k+1)\alpha_{1}})
  \mid^{2}.
\end{eqnarray}
By the inequality \eqref{theorem_res2eq13} and the estimations
\eqref{theorem_res2eq14}, \eqref{theorem_res2eq15} and
\eqref{theorem_res2eq16}, we have
\begin{equation*}
    O(\rho^{-2\alpha_{1}})=\sum\limits_{N:\mid\Lambda_{N}-\mid\gamma\mid^{2}\mid\geq\frac{1}{2}\rho^{2\alpha_{1}}}
    \mid\sum\limits_{\tau=1}^{b_{k}}\sum\limits_{i=1}^{m}\theta_{s}^{\tau i}<\Psi_{N},\Phi_{h_{\tau},i}>\mid^{2}.
\end{equation*}
Therefore, from the decomposition \eqref{theorem_res2eq12} we have
\begin{equation*}
   1-O(\rho^{-2\alpha_{1}})=\sum\limits_{N:\mid\Lambda_{N}-\mid\gamma\mid^{2}\mid<\frac{1}{2}\rho^{2\alpha_{1}}}
   \mid\sum\limits_{\tau=1}^{b_{k}}\sum\limits_{i=1}^{m}\theta_{s}^{\tau i}<\Psi_{N},\Phi_{h_{\tau},i}>\mid^{2}.
\end{equation*}
Since the number of indexes $N$ satisfying
$\mid\Lambda_{N}-\mid\gamma\mid^{2}\mid<\frac{1}{2}\rho^{2\alpha_{1}}$
is less then $\rho^{d-1}$, we have
\begin{equation*}
    1-O(\rho^{-2\alpha_{1}})\leq\rho^{d-1}
    \mid\sum\limits_{\tau=1}^{b_{k}}\sum\limits_{i=1}^{m}\theta_{s}^{\tau i}<\Psi_{N},\Phi_{h_{\tau},i}>\mid^{2}
\end{equation*}
which implies together with the relation \eqref{theorem_res2eq11}
that
\begin{equation}\label{theorem_res2eq17}
    \mid A(N,h_{1},h_{2},\ldots,h_{b_{k}})\cdot \theta_{s}\mid^{2}\geq\frac{1-O(\rho^{-2\alpha_{1}})}{\rho^{d-1}}.
\end{equation}
It follows from the equation \eqref{theorem_res2eq10} and the
estimation \eqref{theorem_res2eq17} that
\begin{equation*}
    \Lambda_{N}=\lambda_{i}+\eta_{s}+
    \frac{O(\rho^{-p\alpha+\frac{d}{4}3^{d}\alpha})}{O(\rho^{-\frac{d-1}{2}})}
\end{equation*}
from which we get the result.
\end{proof}

\section{Asymptotic Formulas for the Eigenvalues in a \\Single
Resonance Domain}

\hspace{.2in} Now, we investigate in detail the eigenvalues of
$L(V)$ in a single resonance domain. Namely, we find the
relation between the eigenvalues of $L(V)$ in a single
resonance domain and the eigenvalues of the Sturm-Liouville
operators. In order the inequalities for $\alpha<\frac{1}{d+20}$,

\begin{equation}\label{alfa1}
2\alpha_2 - \alpha_1 + (d+3)\alpha < 1
\end{equation}
and
\begin{equation}\label{alfa2}
\alpha_2 > 2\alpha_1,
\end{equation}
 to be satisfied, we can choose $\alpha, \,\,
\alpha_1$ and $ \alpha_2$ as follows
$$
\alpha=\frac{1}{d+p},\,\,\,\alpha_1=\frac{p_2}{d+p}
,\,\,\,\alpha_2=\frac{2p_2+1}{d+p},
$$
where $p_2=[\frac{p-5}{3}]-1$ and $[\frac{p-5}{3}]$ is the integer
part of the number $\frac{p-5}{3}$.\newline
 Let $\gamma \in
 V_{\delta}(\rho^{\alpha_{1}})\setminus E_{2}$, $\delta \in
 \frac{\Gamma}{2}\backslash\{e_i\}$, where $\delta$ is minimal in its direction.
Consider the following sets :
\begin{eqnarray}
&B_{1}(\delta)&=\{b:b=n\delta \,,n\in Z \,,\,|b|<\frac{1}{2}\rho^{
\frac{1}{2}\alpha_{2}}\},\nonumber\\
&B_{1}(\gamma)&=\gamma+B_{1}(\delta)=\{\gamma+b:b\in
B_{1}(\delta)\},\nonumber\\
&B_{1}(\gamma,p_{1})&=B_{1}(\gamma)+\Gamma(p_{1}\rho^{\alpha}).\nonumber\hspace{1.28in}
\end{eqnarray}
As before, denote by $h_{s}$,\,\,$s=1,2,...,b_{1}$ the vectors of
$B_{1}(\gamma,p_{1})$, where $b_{1}$ is the number of vectors in
$B_{1}(\gamma,p_{1})$. Then the matrix $C(\gamma,\delta)$ is
defined as:
\begin{equation}\label{appmat}
C(\gamma,\delta)=(c_{ij})=\left[
  \begin{array}{cccc}
    \mid h_{1}\mid^{2} I & V_{h_{1}-h_{2}} & \cdots & V_{h_{1}-h_{b_{1}}}\\
    V_{h_{2}-h_{1}} & \mid h_{2}\mid^{2} I & \cdots& V_{h_{2}-h_{b_{1}}} \\
    \vdots &  &  &  \\
    V_{h_{b_{1}}-h_{1}} & V_{h_{b_{1}}-h_{2}} & \cdots & \mid h_{b_{1}}\mid^{2} I\\
  \end{array}
\right],
\end{equation}
\newline Also we define the matrix
$D(\gamma,\delta)=(c_{ij})$ for $i,j=1,2,...,ma_{1} $, where
$h_{1},h_{2},...,h_{a_{1}}$ are the vectors of
$B_{1}(\gamma,p_1)\bigcap\{\gamma +n \delta : n\in Z\}$, and
$a_{1}$ is the number of vectors in
$B_{1}(\gamma,p_1)\bigcap\{\gamma + n \delta :n\in Z\}$. Clearly
$a_{1}=O(\rho^{ \frac{1}{2}\alpha_{2}})$.

\begin{lemma}\label{lemma421}
a) If $\eta_{s}$ is an eigenvalue of the matrix
$C(\gamma,\delta)$ such that \newline $|\eta_{s}-|h_{s}|^{2}|<M
\,\,\mbox{for} \,\,s=1,2,...,a_{1}$ , then
$$|\eta_{s}-|h_{\tau}|^{2}|>\frac{1}{4}\rho^{\alpha_{2}} \,\,\, \mbox{for} \,\,\, \forall
\tau=a_{1}+1,a_{1}+2,...,b_{1}.$$ b) If $\eta_{s}$ is an
eigenvalue of the matrix $C(\gamma,\delta)$ such that
$|\eta_{s}-|h_{s}|^{2}|<M \,\,\mbox{for}
\,\,s=a_{1}+1,a_{1}+2,...,b_{1}$ then
$$|\eta_{s}-|h_{\tau}|^{2}|>\frac{1}{4}\rho^{\alpha_{2}} \,\,\, \mbox{for} \,\,\, \forall
\tau=1,2,...,a_{1}.$$

\end{lemma}
\begin{proof}
First we prove
\begin{equation}\label{lemma}
||h_{\tau}|^2-|h_s|^2|\geq\frac{1}{3}\rho^{\alpha_{2}}, \hspace{.2in}
\forall s\leq a_1 , \hspace{.2in} \forall \tau> a_1.
\end{equation}
By definition, if $s\leq a_1$ then $h_s=\gamma+n\delta$, where
$|n\delta|<\frac{1}{2}\rho^{\frac{1}{2}\alpha_{2}}+p_1\rho^{\alpha}$.
If $\tau> a_1$ then $h_{\tau}=\gamma+s^{'}\delta+a$, where
$|s^{'}\delta|<\frac{1}{2}\rho^{\frac{1}{2}\alpha_{2}}$,
$a\in\Gamma(p_1\rho^{\alpha})\setminus\delta R$. Therefore
$$
|h_{\tau}|^2-|h_s|^2=2\langle\gamma,a\rangle+2s^{'}\langle\delta,a\rangle+
2s^{'}\langle\gamma,\delta\rangle+|s^{'}\delta|^2+|a|^2-
2n\langle\gamma,\delta\rangle-|n\delta|^2.
$$
Since $\gamma\notin V_a(\rho^{\alpha_2})$, $|a|<p_1\rho^{\alpha}$,
we have
$$
|2\langle\gamma,a\rangle|>\rho^{\alpha_2}-c_9\rho^{2\alpha}.
$$
The relation $\gamma\in V_{\delta}(\rho^{\alpha_1})$ and the
inequalities for $s^{'}$ and $n$ imply that
\begin{eqnarray*}
2s^{'}\langle\gamma,\delta\rangle+2s^{'}\langle\gamma,a\rangle+ |a|^2-
2n\langle\gamma,\delta\rangle=O(\rho^{\frac{1}{2}\alpha_{2}+\alpha_1}),
\newline \\
||s^{'}\delta|^2-|n\delta|^2|<\frac{1}{4}\rho^{\alpha_{2}}
+c_{10}\rho^{\frac{1}{2}\alpha_{2}+\alpha}.
\end{eqnarray*}
Thus (\ref{lemma}) follows from these relations, since
$\frac{1}{2}\alpha_2+\alpha_1<\alpha_2$ and
$\frac{1}{2}\alpha_2+\alpha<\alpha_2$.

The eigenvalues of $D(\gamma,\delta)$ and $C(\gamma,\delta)$ lay
in $M$-neighborhood of the numbers $|h_{k}|^2$ for $k=1,2,...,
a_1$ and for $k=1,2,...,b_1$, respectively. The inequality
(\ref{lemma}) shows that one can enumerate the eigenvalues
$\eta_{s}$ ($s=1,2,...,mb_1$) of $C(\gamma,\delta)$ such that
$\eta_s$ for $s\leq  ma_1$ lay in $M$-neighborhood of the
numbers $|h_k|^2$ for $k\leq a_1$ and $\eta_s$ for $s>ma_1$ lay
in $M$-neighborhood of the numbers $|h_k|^2$ for $k>a_1$.Then by
(\ref{lemma}), we get
\begin{equation}\label{lemma421}
|\eta_s-|h_{\tau}|^2|>\frac{1}{4}\rho^{\alpha_{2}},
\end{equation}
for  $s\leq ma_1$, $ \tau> a_1$ and $s> ma_1$, $ \tau\leq a_1$.

\end{proof}

\begin{theorem}\label{theo421}
Let $\gamma \in V_{\delta}(\rho^{\alpha_{1}})\setminus E_{2}$ and
$|\gamma|\sim \rho$. Then, for any eigenvalue
$\eta_{s}(\gamma)$ of the matrix $C(\gamma,\delta)$ satisfying
$|\eta_{s}-|h_{s}|^{2}|<M$,$\,s=1,2,...,a_{1}$,  there exists
an eigenvalue $\widetilde{\eta}_{k(s)}$ of the matrix
$D(\gamma,\delta)$ such that
$$\eta_{s}=\widetilde{\eta}_{k(s)}+O(\rho^{-\frac{3}{4}\alpha_{2}}),$$
\end{theorem}

\begin{proof} Let $\eta_{s}$ be an
eigenvalue of $C(\gamma,\delta)$, and
$\theta_{s}=(\theta_{s}^1,\theta_{s}^2,...,\theta_{s}^{b_1})_{mb_1\times1}$ be the corresponding
normalized eigenvector, $|\theta_s|=1$, where $\theta_s^{\tau}=(\theta_s^{\tau 1},\theta_s^{\tau 2},\ldots,\theta_s^{\tau m})_{m\times 1}$, $\tau=1,2,\ldots, b_1.$ Denote by $\{e_{\tau,i}\}_{i=1,2,\ldots,m}$ be the set of eigenvectors of $A(\gamma,\delta)$ corresponding to the $\tau$-th
eigenvalue $|h_{\tau}|^{2},$  where
\begin{equation}\label{theorem_res2eq7}
   A= \left[
      \begin{array}{ccc}
        \mid h_{1}\mid^{2}I &  & 0 \\
         & \ddots &  \\
        0 &  & \mid h_{b_{1}}\mid^{2}I \\
      \end{array}
    \right].
\end{equation}
\newline The binding formula for
$C(\gamma,\delta)$ and $A(\gamma,\delta)$,
is :
\begin{equation}\label{bindca}
(\eta_{s}-|h_{\tau}|^{2})\theta_{s}^{\tau i}= \theta_{s}\cdot Be_{\tau,i}
\end{equation}
for any $\tau=1,2,\ldots, b_1$ and $i=1,2,\ldots,m.$  \\

Substituting the orthogonal decomposition
$Be_{\tau,i}=\sum\limits_{k_1=1,2,\ldots,m \atop \xi=1,2,\ldots,b_1}\langle Be_{\tau ,i},e_{\xi,k_1}\rangle
e_{\xi,k_1}$ in equation \ref{bindca}, we get
\begin{eqnarray*}
(\eta_{s}-|h_{\tau}|^{2})\theta_{s}^{\tau i}=
\theta_{s}\cdot  \sum\limits_{k_1=1,2,\ldots,m \atop \xi=1,2,\ldots,b_1}\langle Be_{\tau ,i},e_{\xi,k_1}\rangle
e_{\xi,k_1}       = \sum\limits_{k_1=1,2,\ldots,m \atop \xi=1,2,\ldots,b_1}\langle Be_{\tau ,i},e_{\xi,k_1}\rangle \theta_{s}\cdot e_{\xi,k_1}\nonumber \\
=\sum\limits_{k_1=1,2,\ldots,m \atop \xi=1,2,\ldots,b_1}\langle Be_{\tau ,i},e_{\xi,k_1}\rangle \theta_{s}^{\xi k_1}.
\end{eqnarray*}
It is clear that, 
\begin{displaymath}
\langle Be_{\tau ,i},e_{\xi,k_1}\rangle=\left\{\begin{array}{lll}0&\text{if}&\xi=\tau\\v_{k_1ih_{\xi}-h_{\tau}}&\text{if}&\xi\neq\tau
\end{array}\right.
\end{displaymath}
this implies
\begin{displaymath}
\sum\limits_{k_1=1,2,\ldots,m \atop \xi=1,2,\ldots,b_1}\langle Be_{\tau ,i},e_{\xi,k_1}\rangle=\sum_{\substack{k_1=1,2,\ldots,m \\ \xi=1,\ldots,b_1\\ \xi \neq \tau} }v_{k_1ih_{\xi}-h_{\tau}}
\end{displaymath}
and thus one has
\begin{eqnarray}\label{formulahi}
(\eta_{s}-|h_{\tau}|^{2})\theta_{s}^{\tau i}&=&\sum_{\substack{k_1=1,2,\ldots,m \\ \xi=1,\ldots,b_1\\ \xi \neq \tau} }v_{k_1ih_{\xi}-h_{\tau}} \theta_{s}^{\xi k_1}\nonumber\\&=&\sum_{\substack{k_1=1,2,\ldots,m \\ \xi=1,\ldots,a_1\\ \xi \neq \tau} }v_{k_1ih_{\xi}-h_{\tau}} \theta_{s}^{\xi k_1}+\sum_{\substack{k_1=1,2,\ldots,m \\ \xi=a_1+1,\ldots,b_1\\ \xi \neq \tau} }v_{k_1ih_{\xi}-h_{\tau}} \theta_{s}^{\xi k_1}.
\end{eqnarray}
Now, taking any $\eta_{s} \in [|h_{s}|^{2}-M,|h_{s}|^{2}+M|]
\,\,,\,\,s=1,2,...,a_{1}$ and writing the equation
(\ref{formulahi}) for all $ h_{\tau} \,\,,\tau=1,2,...,a_{1},$ we get
the system of linear algebraic equations:
\begin{eqnarray}\label{systemhi}
(\eta_{s}-|h_{1}|^{2})\theta_{s}^{1i}-\sum_{\substack{k_1=1,2,\ldots,m \\ \xi=1,2,\ldots,a_1\\ \xi \neq 1} }v_{k_1ih_{\xi}-h_{1}} \theta_{s}^{\xi k_1}=\sum_{\substack{k_1=1,2,\ldots,m \\ \xi=a_1+1,\ldots,b_1\\ \xi \neq 1} }v_{k_1ih_{\xi}-h_{1}} \theta_{s}^{\xi k_1}\nonumber \\
(\eta_{s}-|h_{2}|^{2})\theta_{s}^{2i}-\sum_{\substack{k_1=1,2,\ldots,m \\ \xi=1,2,\ldots,a_1\\ \xi \neq 2} }v_{k_1ih_{\xi}-h_{2}} \theta_{s}^{\xi k_1}=\sum_{\substack{k_1=1,2,\ldots,m \\ \xi=a_1+1,\ldots,b_1\\ \xi \neq 2} }v_{k_1ih_{\xi}-h_{2}} \theta_{s}^{\xi k_1}\nonumber 
\\ \vdots \nonumber \\
(\eta_{s}-|h_{a_1}|^{2})\theta_{s}^{a_1i}-\sum_{\substack{k_1=1,2,\ldots,m \\ \xi=1,2,\ldots,a_1\\ \xi \neq a_1} }v_{k_1ih_{\xi}-h_{a_1}} \theta_{s}^{\xi k_1}=\sum_{\substack{k_1=1,2,\ldots,m \\ \xi=a_1+1,\ldots,b_1\\ \xi \neq a_1} }v_{k_1ih_{\xi}-h_{a_1}} \theta_{s}^{\xi k_1}
\end{eqnarray}

Using the binding formula (\ref{bindca}) and the relation 
\eqref{lemma421}, we find an estimation for the right hand side of
the above system. That is,
\begin{align*}
|\sum_{\substack{k_1=1,2,\ldots,m \\ \xi=a_1+1,\ldots,b_1\\ \xi \neq s^{'}} }v_{k_1ih_{\xi}-h_{s^{'}}} \theta_{s}^{\xi k_1}|&=|\sum_{\substack{k_1=1,2,\ldots,m \\ \xi=a_1+1,\ldots,b_1\\ \xi \neq s^{'}} }v_{k_1ih_{\xi}-h_{s^{'}}} \frac{\theta_s \cdot B e_{\xi, k_1}}{(\eta_s-|h_{\xi}|^2)}|\\&\leq \sum_{\substack{k_1=1,2,\ldots,m \\ \xi=a_1+1,\ldots,b_1\\ \xi \neq s^{'}} }|v_{k_1ih_{\xi}-h_{s^{'}}}|\frac{\|\theta_s\|\cdot \|B\|\cdot \|e_{\xi, k_1}\|}{(\eta_s-|h_{\xi}|^2)}\\&\leq 4 \rho^{-\alpha_2}M \sum_{\substack{k_1=1,2,\ldots,m \\ \xi=a_1+1,\ldots,b_1\\ \xi \neq s^{'}} }|v_{k_1ih_{\xi}-h_{s^{'}}}|\\&\leq 4 \rho^{-\alpha_2}M^2\\&=O(\rho^{-\alpha_2})
\end{align*}

for $\forall s^{'}=1,2,\ldots,a_{1}$. Then, the system (\ref{systemhi})
becomes
$$(D(\gamma,\delta)-\eta_{s}I)[\theta_s^1,\theta_s^2,\ldots,\theta_s^{a_1}]^{t}=[O(\rho^{-\alpha_{2}}),O(\rho^{-\alpha_{2}}),...,O(\rho^{-\alpha_{2}})]^{t},$$
or
\begin{equation}\label{ordervi1}
[\theta_s^1,\theta_s^2,\ldots,\theta_s^{a_1}]^{t}=(D(\gamma,\delta)-\eta_{s}I)^{-1}[O(\rho^{-\alpha_{2}}),O(\rho^{-\alpha_{2}}),...,O(\rho^{-\alpha_{2}})]^{t}.
\end{equation}

Using the binding formula (\ref{bindca}) and relation \ref{lemma421}
we have :
\begin{eqnarray*}
\sum_{\substack{i=1,2,\ldots,m \\ \tau=a_1+1,\ldots,b_1} }|\theta_s^{\tau i}|^2&=&\sum\limits_{\substack{i=1,2,\ldots,m \\ \tau=a_1+1,\ldots,b_1} }|\frac{\theta_s \cdot B e_{\tau ,i}}{(\eta_s-|h_{\tau}|^2)}|^2\\&
 =&\sum\limits_{\substack{i=1,2,\ldots,m \\ \tau=a_1+1,\ldots,b_1} }\frac{|B\theta_s \cdot e_{\tau i}|^2}{(\eta_s-|h_{\tau}|^2)^2} \\&\leq&16M^2\rho^{-2\alpha_{2}}\\&=&O(\rho^{-2\alpha_{2}}),
\end{eqnarray*}
and thus, by Parseval's identity we get :
\begin{eqnarray}\label{ordervi2}
\sum\limits_{\substack{i=1,2,\ldots,m \\\nonumber \tau=1,\ldots,a_1} }|\theta_s^{\tau i}|^2&=&\sum\limits_{\substack{i=1,2,\ldots,m \\\nonumber \tau=1,\ldots,b_1} }|\theta_s^{\tau i}|^2-\sum\limits_{\substack{i=1,2,\ldots,m \\\nonumber \tau=a_1+1,\ldots,b_1} }|\theta_s^{\tau i}|^2\\&\geq& 1-O(\rho^{-2\alpha_{2}})\nonumber.
\end{eqnarray}
Now, taking the norm of (\ref{ordervi1}) and using the above
inequality we have :
$$\sqrt{1-O(\rho^{-2\alpha_{2}})}<(\sum_{\substack{i=1,2,\ldots,m \\ \tau=1,\ldots,a_1} }|\theta_s^{\tau i}|^2)^{\frac{1}{2}}\leq
\|(D(\gamma,\delta)-\eta_{s}I)^{-1}\|O(\sqrt{a_{1}}\rho^{-\alpha_{2}})$$
Thus,
$$max|\eta_{s}-\widetilde{\eta}_{k(s)}|^{-1}>\frac{\sqrt{1-O(\rho^{-2\alpha_{2}})}}{\sqrt{a_{1}}\rho^{-\alpha_{2}}},$$
or
$$min|\eta_{s}-\widetilde{\eta}_{k(s)}|=O(\sqrt{a_{1}}\rho^{-\alpha_{2}})=O(\rho^{-\frac{3}{4}\alpha_{2}})$$
where the maximum (minimum) is taken over all $
\widetilde{\eta}_{k(s)}\,\,,s=1,2,...,a_{1}$. So, the result
follows. \end{proof}

\begin{theorem}\label{theo423}
 For any eigenvalue $\widetilde{\eta}_{k}$ of the matrix
$D(\gamma,\delta)$, there exists an eigenvalue $\eta_{s(k)}$ of
the matrix $C(\gamma,\delta)$ such that
$$\eta_{s(k)}=\widetilde{\eta}_{k}+O(\rho^{-\frac{1}{2}\alpha_{2}})$$
\end{theorem}
\begin{proof}  Define the matrix $D'(\gamma,\delta)=(d'_{ij})$ as;
\begin{equation}\label{appmat2}
D'=\left[
  \begin{array}{ccccccccc}
    \mid h_{1}\mid^{2} I & V_{h_{1}-h_{2}} & \cdots & V_{h_{1}-h_{a_{1}}}&0&0&\cdots&0\\
    V_{h_{2}-h_{1}} & \mid h_{2}\mid^{2} I & \cdots& V_{h_{2}-h_{a_{1}}}&0&0&\cdots&0 \\
    \vdots &  &  &  \\
    V_{h_{a_{1}}-h_{1}} & V_{h_{a_{1}}-h_{2}} & \cdots & \mid h_{a_{1}}\mid^{2} I&0&0&\cdots&0\\
    0&0&\cdots&0&\mid h_{a_{1}+1}\mid^{2} I&0&\cdots&0\\
    \vdots&\vdots&\cdots&\vdots&\vdots&\ddots&\cdots&0\\
    0&0&\cdots&0&0&0&\mid h_{b_{1}-1}\mid^{2} I&0\\
    0&0&\cdots&0&0&0&0&\mid h_{b_{1}}\mid^{2} I
  \end{array}
\right]
\end{equation}

Then the spectrum
of the matrix $D'(\gamma,\delta)$ is :
\begin{eqnarray*}
spec(D')&=&spec(D)\bigcup
\{|h_{a_{1}+1}|^{2},|h_{a_{1}+2}|^{2},...,|h_{b_{1}}|^{2}\}
\\
&\equiv&\{\widetilde{\eta}_{1},\widetilde{\eta}_{2},...,\widetilde{\eta}_{ma_{1}},|h_{a_{1}+1}|^{2},|h_{a_{1}+2}|^{2},...,|h_{b_{1}}|^{2}\}.
\end{eqnarray*}
Let us denote by $\Gamma_{k}=(\Gamma_{k}^1,\Gamma_{k}^2,...,\Gamma_{k}^{a_1},0,...,0)_{mb_1\times1}$
the normalized eigenvector corresponding to the $
k$-th eigenvalue
of the matrix $D'$, when $k=1,2,...,m{a_1}$ and by
$\{e_{\tau,i}\}_{i=1,2,\cdots,m}$ the eigenvector corresponding to the
$\tau$-th eigenvalue $|h_{\tau}|^2$ of $D'$, when $\tau=a_{1}+1,a_{1}+2,...,b_{1}$.

Now, using the system from the previous theorem, we have
\begin{eqnarray*}
(D'-\eta_{s}I)[\theta_s^1,\theta_s^2,\ldots,\theta_s^{b_1}]
&=&[(D-\eta_{s}I)[\theta_s^1,\theta_s^2,\ldots,\theta_s^{a_1}],
(|h_{a_{1}+1}|^{2}-\eta_{s})\theta_s^{a_1+1},...,(|h_{b_{1}}|^{2}-\eta_{s})\theta_s^{b_1}]
\\ &=&[O(\rho^{-\alpha_{2}}),...,O(\rho^{-\alpha_{2}}),
(|h_{a_{1}+1}|^{2}-\eta_{s})\theta_s^{a_1+1},...,(|h_{b_{1}}|^{2}-\eta_{s})\theta_s^{b_1}].\hspace{.3in}
\end{eqnarray*}

Taking the inner product of the last system by $\Gamma_k
\,\,,\,\,k=1,2,...,ma_{1}$, using that $D'$ is symmetric and
$D'\Gamma_k=\widetilde{\eta}_{k}\Gamma_k$ we have :
\begin{equation}\label{eigenj}
(\eta_{s(k)}-\widetilde{\eta}_{k})\sum_{j=1}^{a_{1}}\theta_s^j\Gamma_k^j=\sum_{j=1}^{a_{1}}O(\rho^{-\alpha_{2}})\Gamma_k^j,
\end{equation}
where by the Cauchy-Schwarz inequality one has
$$
|\sum_{j=1}^{a_{1}}O(\rho^{-\alpha_{2}})\Gamma_k^j|\leq
\sqrt{\sum_{j=1}^{a_{1}}O(\rho^{-\alpha_{2}})^{2}}\sqrt{\sum_{j=1}^{a_{1}}|\Gamma_k^j|^2}
 \leq
 \sqrt{a_{1}(\rho^{-\alpha_{2}})^{2}}=O(\sqrt{a_{1}}\rho^{-\alpha_{2}}).
$$

Thus, we get
\begin{equation}\label{eigenord}
(\eta_{s(k)}-\widetilde{\eta}_{k})\sum_{j=1}^{a_{1}}\theta_s^j\Gamma_k^j=O(\rho^{-\frac{3}{4}\alpha_{2}}).
\end{equation}

So, we need to show that for any $k=1,2,...,ma_{1}$ there exists
$\,\, \eta_{s(k)},$ such that
\begin{equation}\label{viwi}
|\sum_{j=1}^{a_{1}}{\theta_{s(k)}}^j\Gamma_k^j|=|\langle\theta_{s(k)},\Gamma_k\rangle|>\sqrt{\frac{1-O(\rho^{-\frac{3}{2}\alpha_{2}})}{m\cdot a_{1}}}>c_{23}\rho^{-\frac{1}{4}\alpha_{2}}.
\end{equation}
\newline For this, we consider the decomposition
$\Gamma_k=\sum_{s=1}^{mb_{1}}\langle \Gamma_k,\theta_s\rangle \theta_s$ and the
Parseval's identity
$$1=\sum_{s=1}^{mb_{1}}|\langle \Gamma_k,\theta_s\rangle|^{2}=\sum_{s=1}^{ma_{1}}|\langle \Gamma_k,\theta_s    \rangle|^{2}+\sum_{i=ma_{1}+1}^{mb_{1}}|\langle \Gamma_k,\theta_s\rangle|^{2}$$
First, let us show that
\begin{equation}\label{viwiord}
\sum_{s=ma_{1}+1}^{mb_{1}}|\langle
\Gamma_k,\theta_s\rangle|^{2}=O(\rho^{-\frac{3}{2}\alpha_{2}}).
\end{equation}
Using that, $\Gamma_k=\sum\limits_{\substack{i=1,2,\ldots,m \\ \tau=1,\ldots,a_1} }\langle \Gamma_k, e_{\tau,i}\rangle
e_{\tau,i}$ , the binding formula for $C(\gamma,\delta)$ and
$A(\gamma,\delta)$ ,
the relation \ref{lemma421}, and the Bessel's inequality we obtain the
estimation:
\begin{eqnarray*}
\hspace{1cm}&&\sum_{s=ma_{1}+1}^{mb_{1}}|\langle
\Gamma_k,\theta_s \rangle|^{2}\\&=&\sum_{s=ma_{1}+1}^{mb_{1}}|\langle
\sum\limits_{\substack{i=1,2,\ldots,m \\ \tau=1,\ldots,a_1} } \Gamma_k^{\tau i} e_{\tau,i},\theta_s
\rangle|^{2}\hspace{2.3in}\\&=&\sum_{s=ma_{1}+1}^{mb_{1}}|
\sum\limits_{\substack{i=1,2,\ldots,m \\ \tau=1,\ldots,a_1} }\Gamma_k^{\tau i}\langle e_{\tau,i},\theta_s \rangle|^{2}
=\sum_{s=ma_{1}+1}^{mb_{1}}| \sum\limits_{\substack{i=1,2,\ldots,m \\ \tau=1,\ldots,a_1} }\Gamma_k^{\tau i} \frac{\theta_s \cdot Be_{\tau,i}}{|\eta_{s}-|h_{\tau}|^{2}| }|^{2}\hspace{1.3in}\\
&\leq&16 \sum_{s=ma_{1}+1}^{mb_{1}} \rho^{-2\alpha_{2}}(
\sum\limits_{\substack{i=1,2,\ldots,m \\ \tau=1,\ldots,a_1} }|\Gamma_k^{\tau i} | |\theta_s \cdot Be_{\tau,i}| )^{2}\\&
\leq& \sum_{s=ma_{1}+1}^{mb_{1}}
16\cdot|a_{1}|\cdot m\rho^{-2\alpha_{2}}\left( \sum\limits_{\substack{i=1,2,\ldots,m \\ \tau=1,\ldots,a_1} }   |\Gamma_k^{\tau i}|^{2}    |\theta_s \cdot Be_{\tau,i}|^{2} \right) \\ &\leq&16\rho^{-2\alpha_{2}}|a_{1}|\cdot    m \sum\limits_{\substack{i=1,2,\ldots,m \\ \tau=1,\ldots,a_1} }|\Gamma_k^{\tau i}|^{2}\sum_{s=ma_{1}+1}^{mb_{1}}|\theta_s \cdot Be_{\tau,i}|^{2}\\ &\leq&
16\rho^{-2\alpha_{2}}|a_{1}|\cdot m \sum\limits_{\substack{i=1,2,\ldots,m \\ \tau=1,\ldots,a_1} }|\Gamma_k^{\tau i}|^{2}|Be_{\tau,i}|^{2} \leq
16\rho^{-2\alpha_{2}}|a_{1}|\cdot m \cdot M^2 \sum\limits_{\substack{i=1,2,\ldots,m \\ \tau=1,\ldots,a_1} }|\Gamma_k^{\tau i}|^{2}\hspace{.55in}\\
&\leq &16|a_{1}|m \rho^{-2\alpha_{2}}M^2
 =O(\rho^{-\frac{3}{2}\alpha_{2}}).\hspace{2.5in}
\end{eqnarray*}

Therefore one has
$$\sum_{s=1}^{ma_{1}}|\langle
\Gamma_k, \theta_s\rangle|^{2}=1-O(\rho^{-\frac{3}{2}\alpha_{2}})$$ from
which it follows that there exists $\eta_{s(k)}$ such that
(\ref{viwi}) holds. Dividing both sides of (\ref{eigenord}) by
(\ref{viwi}) we get the result
$$\eta_{s(k)}=\widetilde{\eta_{k}}+O(\rho^{-\frac{1}{2}\alpha_{2}}).$$
\end{proof}
 Now, using the notation
$h_{s}=\gamma-(\frac{s}{2})\delta $ if $s$ is even,
$h_{s}=\gamma+(\frac{s-1}{2})\delta$ if $s$ is odd, for
$s=1,2,...,a_{1}$, (without loss of generality assume that $a_1$
is even) and using the orthogonal decomposition for
 $\gamma\in\frac{\Gamma}{2}$ as $\gamma=\beta+(l+v(\beta))\delta$, where
 $\beta\in H_{\delta}\equiv \{x\in R^{d} : \langle x,\delta
 \rangle=0\},\,\,l\in Z, \,\,v\in[0,1)$ we can write the matrix
 $D(\gamma,\delta)$ as :
 $$D(\gamma,\delta)=|\beta|^{2}I+E(\gamma,\delta),$$
 where $E(\gamma,\delta)$ is: 
 \begin{equation*}\label{appmat2}
 E=\left[
   \begin{array}{ccccccccc}
   \left((l+v)^2 \mid \delta \mid^{2}\right) I & V_{\delta} &  V_{-\delta} &\cdots&V_{\frac{a_1}{2}\delta}\\
     V_{-\delta} & \left((l-1+v)^2 \mid \delta \mid^{2}\right) I &V_{-2\delta} & \cdots&V_{\left(\frac{a_1}{2}-1\right)\delta} \\V_{\delta}&V_{2\delta}&\left((l+1+v)^2 \mid \delta \mid^{2}\right) I&\ldots&V_{\left(\frac{a_1}{2}+1\right)\delta} \\
     \vdots & \vdots  &\vdots&\vdots&\vdots \\
     V_{-\frac{a_1}{2}\delta} &\cdot&\cdot&\ldots& \left((l-\frac{a_1}{2}+v)^2 \mid \delta \mid^{2}\right) I
   \end{array}
 \right]
 \end{equation*}
Denote $n_k=-\frac{k}{2}$ if $k$ is even, $n_k=\frac{k-1}{2}$ if
$k$ is odd. The system
\newline $\{e^{i(n_k+v)s^{'}}: k=1,2,...\}$ is
a basis in $L_2[0,2\pi]$. Let $T(\gamma,\delta)\equiv
T(P(s^{'}),\beta)$ be the operator in $\ell_2$ corresponding to the
Sturm-Liouville operator $T$, generated by
\begin{equation}\label{sturmeqn}
-|\delta|^{2}Y''+P(s^{'})Y=\mu Y
\end{equation}
\begin{eqnarray*}
Y(s^{'}+2\pi)=e^{i2\pi v(\beta)}Y,
\end{eqnarray*}
where $P(s^{'})=(p_{ij}(s^{'}))$ for $p_{ij}(s^{'})=\sum\limits
_{k=1}^{\infty}v_{ijn_k\delta}e^{in_ks^{'}},$ and
$v_{ijn_k\delta}=(v_{ij}(x),\frac{1}{|A_{n_k\delta}|}\sum_{\alpha\in
A_{n_k\delta}}e^{i\langle\alpha,s^{'}\rangle})$, $s=\langle
x,\delta\rangle.$ It means that $T(\gamma,\delta)$ is the infinite
matrix $(Te^{i(l+n_k+v)s^{'}}, e^{i(l+n_m+v)s^{'}})$, $k,m=1,2,\ldots.$
\begin{theorem}\label{lasttheorem}
For every eigenvalue $\eta_{s}$ of the Sturm-Liouville operator
$T(\gamma,\delta),$ there exists an eigenvalue
$\widetilde{\eta_{s}}$ of the matrix $E(\gamma,\delta)$ such that
$$\eta_{s}=\widetilde{\eta_{s}}+O(\rho^{-\frac{3}{4}\alpha_{2}}).$$
\end{theorem}
\begin{proof}
Decompose the infinite matrix $T(\gamma,\delta)$ as
$T(\gamma,\delta)=\widetilde{A}+\widetilde{B}$, where the matrix
$\widetilde{A}$ is defined by
\begin{equation}\label{theorem_res2eq7}
   \widetilde{A}= \left[
      \begin{array}{ccc}
        \left((l+v)^2 \mid \delta \mid^{2}\right) I  &  & 0 \\& \left((l-1+v)^2 \mid \delta \mid^{2}\right) I  \\
         & \ddots &  \\
        0 &  &  \left((l-\frac{a_1}{2}+v)^2 \mid \delta \mid^{2}\right) I\\
      \end{array}
    \right]
\end{equation}
and $ \widetilde{B}=T(\gamma,\delta)- \widetilde{A}.$
Let ${\eta_{s}}$ be an eigenvalue of $T(\gamma,\delta)$, and
$\Theta_{s}=(\Theta_{s}^1,\Theta_{s}^2,\Theta_{s}^3,...) $ be the corresponding
normalized eigenvector, that is $T\Theta_{s}=\eta_{s}\Theta_{s}.$ Denote by
$\{e_{\tau i}\}_{i=1}^m=(0,0,...,1,0,...)$ the $\tau$-th eigenvector of the matrix
$\widetilde{A}$. Then, the corresponding $\tau$-th eigenvalue of
$\widetilde{A}$ is $|(\tau'+v)\delta|^{2}$, that is $
\widetilde{A}e_{\tau i}=|(\tau'+v)\delta|^{2}e_{\tau i},$ where
$\tau'=l-\frac{\tau}{2}$ if $\tau$ is even, $\tau'=l+\frac{\tau-1}{2}$ if $\tau$ is
odd, for $\tau=1,2,...$.\newline One can easily verify that
\begin{equation}\label{star}
(\eta_{s}-|(\tau'+v)\delta|^{2})\theta_{s}^{\tau i}=\langle
\Theta_{s},\widetilde{B}e_{\tau i}\rangle
\end{equation}
 and using the orthogonal decomposition
$\widetilde{B}e_{\tau i}=\sum\limits
_{j=1}^{m}\sum\limits
_{k=1}^{\infty}\langle\widetilde{B}e_{\tau i},e_{k,j}\rangle
e_{k,j}$, we get
$$(\eta_{s}-|(\tau'+v)\delta|^{2})\Theta_s^{\tau i}=\sum\limits
_{j=1}^{m}\sum\limits
_{k=1}^{\infty}\langle\widetilde{B}e_{\tau i},e_{k,j}\rangle\Theta_s^{k j}
$$
and since
$\langle\widetilde{B}e_{\tau,i},e_{k,j}\rangle=v_{ji(n_k-n_{\tau})\delta}$ for  $k\neq \tau$,
\begin{equation}\label{lastest}
(\eta_{s}-|(\tau'+v)\delta|^{2})\Theta_{s}^{\tau i}-\sum_{j=1}^{m}\sum_{k=1}^{a_{1}}v_{ji(n_k-n_{\tau})\delta}\Theta_s^{k j}
=\sum_{j=1}^{m}\sum_{k=a_{1}+1}^{\infty}v_{ji(n_k-n_{\tau})\delta}\Theta_s^{k j}.
\end{equation}

Now take any eigenvalue $\eta_{s}$ of $T(\gamma,\delta)$,
satisfying $|\eta_{s}-|(i'+v)\delta|^{2}|< sup|P(s^{'})|$ for
$s=1,2,...,\frac{a_1}{2}$, where $i'=l-\frac{s}{2}$ if $s$ is
even, $i'=l+\frac{s-1}{2}$ if $s$ is odd. The relations $\gamma\in
V_{\delta}(\rho^{\alpha_1})$ $(\delta\neq e_i)$ and $\gamma=\beta+(l+v)\delta$,
$\langle \beta, \delta \rangle=0$ imply
$$
|2\langle \gamma, \delta \rangle+|\delta|^2|=
|(l+v)|\delta|^2+|\delta|^2|<\rho^{\alpha_1},\hspace{.2in}
|l|<c_{12}\rho^{\alpha_1}.
$$
Therefore using the definition of $i'$ and $\tau'$, we have
$$
|(i'+v)\delta|<\frac{|a_1\delta|}{4}+c_{23}\rho^{\alpha_1},
$$
for $s=1,2,...\frac{a_1}{2}$ and
$$
|(\tau'+v)\delta|>\frac{|a_1\delta|}{2}-c_{24}\rho^{\alpha_1},
$$
for $\tau>a_1.$ Since $|a_1|>c_{25}\rho^{\frac{\alpha_2}{2}}$ and
$\alpha_2>2\alpha_1,$ we have
\begin{eqnarray}
||(i'+v)\delta|^{2}-|(\tau
'+v)\delta|^{2}|>c_{26}\rho^{\alpha_{2}},
\end{eqnarray}
for $s\leq \frac{a_1}{2}$, $\tau>a_1$, which implies
\begin{eqnarray}\label{mest}
|\eta_{s}-|(\tau'+v)\delta|^{2}|\nonumber=
||\eta_{s}-|(i'+v)\delta|^{2}|-||(\tau'+v)\delta|^{2}|-|(i'+v)\delta|^{2}||
>c_{17}\rho^{\alpha_{2}},
\end{eqnarray}
for $s=1,2,...\frac{a_1}{2}\,\,,\,\,\tau>a_1$. \newline Thus,
\begin{eqnarray}\label{46}
|\sum_{j=1}^{m}\sum_{k=a_{1}+1}^{\infty}v_{ij(n_k-n_{\tau})\delta}\Theta_s^{kj}|=\sum_{j=1}^{m}\sum_{k=a_{1}+1}^{\infty}
|v_{ij(n_k-n_{\tau})\delta}| \left|\frac{\langle \Theta
_{s},\widetilde{B}
e_{k,j}\rangle}{\eta_{s}-|(k^{'}+v)\delta|^{2}}\right| \nonumber \\
\leq \sum_{j=1}^{m}\sum_{k=a_{1}+1}^{\infty}|v_{ij(n_k-n_{\tau})\delta}| \frac{\|
\Theta_{s}\|\|\widetilde{B}\|
\|e_{k,j}\|}{|\eta_{s}-|(k^{'}+v)\delta|^{2}|}\leq M\rho^{-\alpha_{2}}
\sum_{j=1}^{m}\sum_{k=a_{1}+1}^{\infty}|v_{ij(n_k-n_{\tau})\delta}|\\
\leq c_{18}\rho^{-\alpha_{2}},\nonumber
\end{eqnarray}
since $\|\widetilde{B}\|\leq M $. Indeed, $\widetilde{B}$
corresponds to the operator $P:Y\rightarrow P(s^{'})Y $ in
$L_{2}[0,2\pi]$, which has norm $sup|P(s^{'})|\leq M.$
 Therefore writing the equation (\ref{lastest}) for all $ \tau=1,2,...,a_{
 1}$, and using
(\ref{46}) we get the following system
\begin{eqnarray}\label{47}
(E(\gamma,\delta)-\eta_{s}I)[\Theta_{s}^1,\Theta_{s}^2,...,\Theta_{s}^{a_{1}}]
=[O(\rho^{-\alpha_{2}}),O(\rho^{-\alpha_{2}}),...,O(\rho^{-\alpha_{2}})].
\end{eqnarray}

Using that, $\Theta_{s}=\sum\limits
_{i=1}^{\infty}\Theta_{s}^{\tau}e_{\tau,i} $, the formula
(\ref{star}) and (\ref{mest}), we have
$$\sum_{j=a_{1}+1}^{\infty}|\Theta_{s}^{\tau i}|^{2}=\sum_{j=a_{1}+1}^{\infty}
|\frac{\langle \Theta_{s},\widetilde{B}
e_{\tau,i}\rangle}{\eta_{s}-|(\tau^{'}+v)\delta|^{2}}|^{2}=O(\rho^{-2\alpha_{2}})$$
and thus
,$$\sum_{i=1}^{a_{1}}|\Theta_{s}^{\tau i}|^{2}=1-O(\rho^{-2\alpha_{2}}).$$
Taking the norm of the vector (see (\ref{47}))
\begin{eqnarray*}
[\Theta_{s}^1,\Theta_{s}^2,...,\Theta_{s}^{a_{1}}] =(E(\gamma,\delta)-\eta_{s}I)^{-1}
[O(\rho^{-\alpha_{2}}),...,O(\rho^{-\alpha_{2}})],
\end{eqnarray*}
we get
 $$
 \sqrt{\frac{1-O(\rho^{-2\alpha_{2}})}{m}}=
\|(E(\gamma,\delta)-\eta_{s}I)^{-1}\|O(\sqrt{a_{1}}\rho^{-\alpha_{2}})
$$
or
$$
\min_j|\eta_{s}-\widetilde{\eta}_{\tau}|=\frac{O(\sqrt{a_{1}}\rho^{-\alpha_{2}})\cdot
\sqrt{m}}
{\sqrt{1-O(\rho^{-2\alpha_{2}})}}=O(\rho^{-\frac{3}{4}\alpha_{2}}),
$$
where the minimum is taken over all eigenvalues $\,\,
\widetilde{\eta}_{\tau}$ of the matrix $E(\gamma,\delta)$. Thus, the
result follows. 
\end{proof}

\begin{theorem}(\textbf{Main result})\label{maintheorem}
For every $\beta\in H_{\delta}$ ,$\,|\beta|\sim \rho$ and for
every eigenvalue $\eta_{s}(v(\beta))$ of the Sturm-Liouville
operator $T(\gamma,\delta),$ there is an eigenvalue $\Upsilon_N$
of the operator $L(V)$  satisfying
$$
\Upsilon_N=|\beta|^{2}+\eta_s+\lambda_i+O(\rho^{-\frac{1}{2}\alpha_{2}}),
$$
where $\lambda_i$ is a given eigenvalue of the matrix $V_0$.

\end{theorem}
\begin{proof}
From Theorem \ref{lasttheorem} and the definition of
$E(\gamma,\delta),$  there exists an eigenvalue
$\widetilde{\eta}_{l(s)}$ of the matrix $D(\gamma,\delta),$
where $\gamma$ has a decomposition
$\gamma=\beta+(l+v(\beta))\delta,$  satisfying
$\widetilde{\eta}_{l(s)}=|\beta|^{2}+\eta_{s}+O(\rho^{-\frac{3}{4}\alpha_{2}}).$
Therefore, the result follows from Theorem \ref{theo423} and
Theorem \ref{theorem_res2}. 
\end{proof}

\end{document}